\documentclass[10pt,draft]{article}
\usepackage{amsmath,amscd}
\usepackage{amssymb,latexsym,amsthm}
\usepackage{color}
\usepackage[spanish,english]{babel}
\usepackage{amssymb}
\usepackage{color}
\usepackage{amsmath,amsthm,amscd}
\usepackage[latin1]{inputenc}
\usepackage{lscape}
\usepackage{fancyhdr}
\usepackage{amsfonts}
\usepackage{pb-diagram}
\usepackage{hyperref}
\usepackage{float}
\numberwithin{equation}{section}
\newtheorem{theorem}{Theorem}[section]

\newtheorem{definition}[theorem]{Definition}
\newtheorem{proposition}[theorem]{Proposition}
\newtheorem{lemma}[theorem]{Lemma}

\newtheorem{corollary}[theorem]{Corollary}

\theoremstyle{definition}

\newtheorem{example}[theorem]{Example}
\newtheorem{examples}[theorem]{Examples}

\newtheorem{remark}[theorem]{Remark}

\newcommand{\cA}{\mbox{${\cal A}$}}

\newcommand{\cD}{\mbox{${\cal D}$}}

\newcommand{\cG}{\mbox{${\cal G}$}}

\newcommand{\cJ}{\mbox{${\cal J}$}}

\newcommand{\cO}{\mbox{${\cal O}$}}

\newcommand{\cU}{\mbox{${\cal U}$}}

\newcommand{\cW}{\mbox{${\cal W}$}}

\newcommand{\KK}{\ensuremath{\mathbb{K}}}

\newcommand{\NN}{\ensuremath{\mathbb{N}}}

\newcommand{\Q}{\ensuremath{\mathbb{Q}}}

\title{\textbf{Center of skew $ PBW $ extensions}}
\author{José Oswaldo Lezama Serrano\\
\texttt{jolezamas@unal.edu.co}
\\Helbert Javier Venegas Ramírez\\
\texttt{hjvenegasr@unal.edu.co}
\\ Seminario de Álgebra Constructiva - SAC$^2$\\ Departamento de Matemáticas\\ Universidad Nacional de
Colombia, Sede Bogotá}
\date{}
\begin{document}
\maketitle
\begin{abstract}
\noindent In this paper we compute the center of many noncommutative algebras that can be
interpreted as skew $PBW$ extensions. We show that, under some natural assumptions on the
parameters that define the extension, either the center is trivial, or, it is of polynomial type.
As an application, we provided new examples of noncommutative algebras that are cancellative.

\bigskip

\noindent \textit{Key words and phrases.}  Center of an algebra, quantum algebras, skew $PBW$
extensions, Zariski cancellation problem.

\bigskip

\noindent 2010 \textit{Mathematics Subject Classification.} Primary: 16U70. Secondary: 16S36,
16S38.
\end{abstract}
\section{Introduction}

The center and the centralizers are natural commutative subalgebras that play an important role in
representation theory and in the general theory of rings and algebras. Recently these subalgebras
were used by Bell and Zhang to investigate the Zariski Cancellation Problem (ZCP) for
noncommutative algebras (\cite{BellZhang}); some other authors have also studied related questions
as the automorphism problem and the Dixmier conjecture applying the description of the center (see
\cite{BackelinE}, \cite{Dixmier2}, \cite{Richard}, \cite{Tang}). The center of many algebras coming
from mathematical physics has been computed in the last years, among the examples are:  The Weyl
algebra, the quantum Weyl algebra of Maltsiniotis, the quantized Weyl algebra $A_n^{Q,\Gamma}(K)$,
the quantum symplectic space $\cO_q(\mathfrak{sp}(K^{2n}))$, some universal enveloping algebras of
Lie algebras, the Jordan plane, the quantum plane (see \cite{Zhangetal3}, \cite{Levandovskyy},
\cite{McConnell}, \cite{Shirikov}, \cite{Tang}, \cite{Hechun}).

In this paper we will compute the center and some central elements for a wide family of algebras
and rings that can be interpreted as skew $ PBW $ extensions. The skew $PBW$ extensions were
introduced by Lezama and Gallego in \cite{LezamaGallego} and represent a generalization of $PBW$
(Poincaré-Birkhoff-Witt) extensions defined by Bell and Goodearl (\cite{Bell}). In addition to the
examples mentioned above, some other prominent algebras that can be covered by the skew $PBW$
extensions are: The algebra of $q$-differential operators, the algebra of shift operators, the
algebra of linear partial $q$-differential operators, skew polynomial rings of derivation type and
Ore extension of derivation type, the additive and multiplicative analogues of the Weyl algebra,
the quantum algebra $U'(so(3,K))$, the $3$-dimensional skew polynomial algebra, the dispin algebra,
the Woronowicz algebra, the $q$-Heisenberg algebra, the Witten's deformation of
$\mathcal{U}(\mathfrak{sl}(2,K)$, algebras diffusion type, some quadratic algebras.

The technique that we use in this paper for computing the center is very simple: We interpret the
algebras as skew $PBW$ extensions (when this apply) and we use the parameters and relations that
define the extension (Definition \ref{definition1.2.1}), then, taking an element $f$ of the center
of the algebra we deduce some formulae from the basic equation $x_if=fx_i$, where $x_i$ is any of
the variables that define the extension. From these formulae we compute either the center, or, some
central key subalgebras when the description of the center is not completely determined.

The paper is organized as follows: In the first section we recall the definition of the skew $PBW$
extensions, and we present some of its basic properties used in the computations involved in the
paper; moreover, we review some examples of skew $ PBW $ extensions whose center is well-known in
the literature. In Section 2 we describe the center of some general subclasses of skew $PBW$
extensions, grouped according to some conditions imposed to the parameters that define the
extension. The computation of the center for the algebras covered in this section probably is new.
In Section 3 we consider the center of other particular examples of skew $PBW$ not included in the
subclasses of the second section. The results here are probably new too. In the last section we
summarize the results of the two previous sections and, the most important thing, we apply these
results to give new examples of noncommutative algebras that are cancellative in the sense of the
Zariski problem. In this paper, $K$ denotes a field with $char(K)=0$.

\subsection{Definition and basic properties of skew $PBW$ extensions}

\begin{definition}[\cite{LezamaGallego}]\label{gpbwextension}
Let $R$ and $A$ be rings. We say that $A$ is a \textit{skew $PBW$ extension of $R$} $($also called
a $\sigma-PBW$ extension of $R$$)$ if the following conditions hold:
\begin{enumerate}
\item[\rm (i)]$R\subseteq A$.
\item[\rm (ii)]There exist finitely many elements $x_1,\dots ,x_n\in A$ such $A$ is a left $R$-free module with basis
\begin{center}
${\rm Mon}(A):= \{x^{\alpha}=x_1^{\alpha_1}\cdots x_n^{\alpha_n}\mid \alpha=(\alpha_1,\dots
,\alpha_n)\in \mathbb{N}^n\}$, with $\mathbb{N}:=\{0,1,2,\dots\}$.
\end{center}
The set $Mon(A)$ is called the set of standard monomials of $A$.
\item[\rm (iii)]For every $1\leq i\leq n$ and $r\in R-\{0\}$ there exists $c_{i,r}\in R-\{0\}$ such that
\begin{equation}\label{sigmadefinicion1}
x_ir-c_{i,r}x_i\in R.
\end{equation}
\item[\rm (iv)]For every $1\leq i,j\leq n$ there exists $c_{i,j}\in R-\{0\}$ such that
\begin{equation}\label{sigmadefinicion2}
x_jx_i-c_{i,j}x_ix_j\in R+Rx_1+\cdots +Rx_n.
\end{equation}
Under these conditions we will write $A:=\sigma(R)\langle x_1,\dots ,x_n\rangle$.
\end{enumerate}
\end{definition}
Associated to a skew $PBW$ extension $A=\sigma(R)\langle x_1,\dots ,x_n\rangle$, there are $n$
injective endomorphisms $\sigma_1,\dots,\sigma_n$ of $R$ and $\sigma_i$-derivations, as the
following proposition shows.
\begin{proposition}\label{sigmadefinition}
Let $A$ be a skew $PBW$ extension of $R$. Then, for every $1\leq i\leq n$, there exist an injective
ring endomorphism $\sigma_i:R\rightarrow R$ and a $\sigma_i$-derivation $\delta_i:R\rightarrow R$
such that
\begin{center}
$x_ir=\sigma_i(r)x_i+\delta_i(r)$,
\end{center}
for each $r\in R$.
\end{proposition}
\begin{proof}
See \cite{LezamaGallego}, Proposition 3.
\end{proof}

A particular case of skew $PBW$ extension is when all derivations $\delta_i$ are zero. Another
interesting case is when all $\sigma_i$ are bijective and the constants $c_{ij}$ are invertible. We
recall the following definitions.
\begin{definition}\label{sigmapbwderivationtype}
Let $A$ be a skew $PBW$ extension.
\begin{enumerate}
\item[\rm (a)]
$A$ is quasi-commutative if the conditions {\rm(}iii{\rm)} and {\rm(}iv{\rm)} in Definition
\ref{gpbwextension} are replaced by
\begin{enumerate}
\item[\rm (iii')]For every $1\leq i\leq n$ and $r\in R-\{0\}$ there exists $c_{i,r}\in R-\{0\}$ such that
\begin{equation}
x_ir=c_{i,r}x_i.
\end{equation}
\item[\rm (iv')]For every $1\leq i,j\leq n$ there exists $c_{i,j}\in R-\{0\}$ such that
\begin{equation}
x_jx_i=c_{i,j}x_ix_j.
\end{equation}
\end{enumerate}
\item[\rm (b)]$A$ is bijective if $\sigma_i$ is bijective for
every $1\leq i\leq n$ and $c_{i,j}$ is invertible for any $1\leq i,j\leq n$.
\item[\rm (c)]$A$ is constant if for every $1\leq i\leq n$ and $r\in R$,
\begin{equation}
x_ir=rx_i,
\end{equation}
i.e., $\sigma_i=i_R$ and $\delta_i=0$ for every $1\leq i\leq n$.
\end{enumerate}
\end{definition}

If $A=\sigma(R)\langle x_1,\dots,x_n\rangle$ is a skew $PBW$ extension of the ring $R$, then, as
was observed in Proposition \ref{sigmadefinition}, $A$ induces injective endomorphisms
$\sigma_k:R\to R$ and $\sigma_k$-derivations $\delta_k:R\to R$, $1\leq k\leq n$. Moreover, from the
Definition \ref{gpbwextension}, there exists a unique finite set of constants $c_{ij}, d_{ij},
a_{ij}^{(k)}\in R$, $c_{ij}\neq 0$, such that
\begin{equation}\label{equation1.2.1}
x_jx_i=c_{ij}x_ix_j+a_{ij}^{(1)}x_1+\cdots+a_{ij}^{(n)}x_n+d_{ij}, \ \text{for every}\  1\leq
i<j\leq n.
\end{equation}

\begin{definition}\label{definition1.2.1}
Let $A=\sigma(R)\langle x_1,\dots,x_n\rangle$ be a skew $PBW$ extension. $R$, $n$,
$\sigma_k,\delta_k, c_{ij}$, $d_{ij}, a_{ij}^{(k)}$, with $1\leq i<j\leq n$, $1\leq k\leq n$,
defined as before, are called the parameters of $A$.
\end{definition}

\begin{definition}\label{1.1.6}
Let $A$ be a skew $PBW$ extension of $R$.
\begin{enumerate}
\item[\rm (i)]For $\alpha=(\alpha_1,\dots,\alpha_n)\in \mathbb{N}^n$,
$|\alpha|:=\alpha_1+\cdots+\alpha_n$.
\item[\rm (ii)]For $X=x^{\alpha}\in Mon(A)$,
$\exp(X):=\alpha$ and $\deg(X):=|\alpha|$.
\item[\rm (iii)]Let $0\neq f\in A$, $t(f)$ is the finite
set of terms that conform $f$, i.e., if $f=c_1X_1+\cdots +c_tX_t$, with $X_i\in Mon(A)$ and $c_i\in
R-\{0\}$, then $t(f):=\{c_1X_1,\dots,c_tX_t\}$.
\item[\rm (iv)]Let $f$ be as in {\rm(iii)}, then $\deg(f):=\max\{\deg(X_i)\}_{i=1}^t.$
\end{enumerate}
\end{definition}

 The next result is an important property of skew $ PBW $ extensions that will be used later.

 \begin{theorem}[\cite{Oswaldo}]\label{filteredskew}
    Let $ A $ be an arbitrary skew $PBW$ extension of the ring $ R $. Then, $ A $ is a filtered ring with filtration given by
    $$F_{m}:=
    \begin{cases}
    R, & \text{ if } m=0\\
    \{f \in A| deg(f) \leq m\}, & \text{  if } m \geq 1,
    \end{cases}$$
    and the graded ring $ Gr(A) $ is a quasi-commutative skew $ PBW $ extension of $ R
    $. If the parameters that define $A$ are as in Definition \ref{definition1.2.1}, then the
    parameters that define $Gr(A)$ are $R$, $n$, $\sigma_k, c_{ij}$, with $1\leq i<j\leq n$, $1\leq k\leq
    n$. Moreover, if $ A $ is bijective, then $ Gr(A) $ is bijective.
     \end{theorem}
\begin{remark}\label{remark1.7}
Additionally in this paper we will assume that if $A=\sigma(R)\langle x_1,\dots,x_n\rangle$ is a
skew $PBW$ extension, then $R$ is $K$-algebra such that $\sigma_i(k)=k$ and $\delta_i(k)=0$ for
every $k\in K$ and $1\leq i\leq n$. Therefore, $A$ is also a $K$-algebra. In particular, if $R=K$,
then $A$ is constant.
\end{remark}

\subsection{Examples}\label{examples}
In this subsection we review some examples of skew $ PBW $ extensions whose center has been
computed before in the literature. For more details about the precise definition of the next
algebras and its homological properties see \cite{Oswaldo} and \cite{Reyes}. Recall that $K$ is a
field with $char(K)=0$.

\begin{example}[\textbf{$ PBW $ extensions}]\label{PBW}
Any $ PBW $ extension (see \cite{Bell}) is a bijective skew $ PBW $ extension since in this case $
\sigma_{i}=i_{R} $ for each $ 1 \leq i \leq n $, and $ c_{i,j}=1 $ for every $ 1 \leq i,j \leq n $.
Thus, for $ PBW $ extensions we have $ A=i(R)\langle x_{1}, \dots, x_{n}\rangle $. Some examples of
$ PBW $ extensions are the following:
\begin{enumerate}
\item[(a)] The \textit{usual polynomial algebra} $ A=K[t_{1},\dots,t_{n}] $, $ t_it_j=t_jt_i $, so $ Z(A)=K[t_{1},\dots,t_{n}] $.
\item[(b)] The \textit{Weyl algebra} $ A_{n}(K):=K[t_{1},\dots,t_{n}][x_{1};\partial/\partial t_{1}] \cdots [x_{n};\partial / \partial t_{n}] $.
The \textit{extended Weyl algebra} $ B_{n}(K):=K(t_{1},\dots,t_{n})[x_{1};\partial/\partial t_{1}]
\cdots [x_{n};\partial / \partial t_{n}] $, where $ K(t_{1},\dots,t_{n}) $ is the field of
fractions of $ K[t_{1},\dots,t_{n}] $. It is known that $ Z(A_n(K))=K=Z(B_{n}(K))$ (see
\cite{Dixmier}).
\item[(c)] The \textit{universal enveloping algebra} of a finite dimensional Lie algebra $ \mathcal{U}(\mathcal{G}) $.
In this case, $ x_{i}r-rx_{i}=0 $ and $ x_{i}x_{j}-x_{j}x_{i}=[x_{i},x_{j}]\in
\mathcal{G}=K+Kx_{1}+\dots+Kx_{n} $, for any $ r \in K $ and $ 1 \leq i,j\leq n $. The center of
some of these algebras was studied in \cite{Levandovskyy} and \cite{McConnell}:
\begin{itemize}
    \item \textit{The universal enveloping algebra $\mathcal{U}(\mathfrak{sl}(2,K))$ of the Lie algebra $\mathfrak{sl}(2,K)$} is the $ K- $algebra generated
    by the variables $x, y, z$ subject to the relations
    \begin{align*}
    [x, y] = z,\  [x, z]=-2x,\  [y, z] = 2y.
    \end{align*}
    Then $ Z(\mathcal{U}(\mathfrak{sl}(2,K)))=K[4xy+z^2-2z] $.
    \item  $ \mathcal{U}(\mathfrak{so}(3,K)) $ is the $ K- $algebra generated by the variables $x, y, z$ subject to
        the relations
        \begin{align*}
        [x, y] = z,\  [x, z] = -y,\  [y, z] = x.
    \end{align*}
So, $ Z(\mathcal{U}(\mathfrak{so}(3,K)))=K[x^2+y^2+z^2] $.
\item Let $ \mathcal{G} $ be a three-dimensional completely solvable Lie algebra with basis $ x,y,z $ such that
\begin{align*}
 [y,x]=q_1y,\  [z,x]=q_2z, \ [z,y]=0 .
\end{align*}
Then, if either $ q_1,q_2\in K-\Q $, or, $ q_1,q_2\in \Q $ and $ q_1q_2>0 $, then $
Z(U(\mathcal{G}))=K $. If $ q_1,q_2\in \Q $ and $ q_1q_2<0 $ then $ Z(U(\mathcal{G}))\supsetneq K
$.
\end{itemize}
\end{enumerate}
\end{example}
\begin{example}[\textbf{Ore extensions of bijective type}]\label{ore}
Any \textit{skew polynomial ring} $ R[x;\sigma,\delta] $ of bijective type, i.e., with $ \sigma $ bijective, is a bijective skew $ PBW $ extension. In this case we have $ R[x;\sigma,\delta]\cong \sigma(R)\langle x \rangle $. If additionally $ \delta=0 $, then $ R[x;\sigma] $ is quasi-commutative.
More generally, let $ R[x_{1};\sigma_{1},\delta_{1}]\cdots[x_{n};\sigma_{n},\delta_{n}] $ be an \textit{iterated skew polynomial ring of bijective type}, i.e., the following conditions hold:
\begin{itemize}
\item for $ 1\leq i \leq n $, $ \sigma_{i} $ is bijective;
\item for every $ r \in R $ and $ 1 \leq i \leq n $, $ \sigma_{i}(r) \in R$,  $ \delta_{i}(r) \in R$;
\item for $ i < j $, $ \sigma_{j}(x_{i})=cx_{i}+d $, with $ c,d\in R $ and $ c $ has a left inverse;
\item for $ i<j $, $ \delta_{j}(x_{i})\in R+Rx_{1}+\cdots+Rx_{n} $,
\end{itemize}
then $ R[x_{1};\sigma_{1},\delta_{1}]\cdots[x_{n};\sigma_{n},\delta_{n}]  $ is a bijective skew $ PBW $ extension. Under those conditions we have
$$ R[x_{1};\sigma_{1},\delta_{1}]\cdots[x_{n};\sigma_{n},\delta_{n}] \cong \sigma(R)\langle x_{1},\dots,x_{n} \rangle. $$
Some concrete examples of Ore algebras of bijective type are the following:
\begin{enumerate}
\item [\rm (a)]\label{D_qsigma} \textit{The algebra of $q$-differential operators $D_{q,h}[x,y]$}: Let
$q,h\in K, q\neq 0$; consider $K[y][x;\sigma,\delta]$, $\sigma(y):=qy$ and $\delta(y):=h$. By
definition of skew polynomial ring we have $xy=\sigma(y)x+\delta(y)=qyx+h$, and hence $xy-qyx=h$.
Therefore, $D_{q,h}\cong \sigma(K[y])\langle x\rangle$. Then, $ Z(D_{q,h})=K $ if $ q $ is not a
root of unity, and $Z(D_{q,h})= K[x^n,y^n] $ if $ q $ is a primitive $ n $-th root of unity (see
\cite{Zhangetal3}).
\item [\rm (b)]\label{shift} \textit{The algebra of shift operators $S_h$}: Let
$h\in K$. The algebra of shift operators is defined by $S_h:=K[t][x_h;\sigma_h,\delta_h]$, where
$\sigma_h(p(t)):=p(t-h)$, and $\delta_h:=0$. Thus, $S_h\cong \sigma(K[t])\langle x_h\rangle$ and it
is easy to check that $ Z(S_h)=K$ (see also \cite{David}).
\item [\rm (c)]\label{mixed} The \textit{mixed algebra $D_h$}: Let
$h\in K$. The algebra $D_h$ is defined by
$D_h:=K[t][x;i_{K[t]},\frac{d}{dt}][x_h;\sigma_h,\delta_h]$, where $\sigma_h(x):=x$. Then $D_h\cong
\sigma(K[t])\langle x,x_h\rangle$ and $Z(D_h)=K$ (\cite{David}).
\end{enumerate}
\end{example}

\begin{example}[\textbf{Quantum algebras}]\label{Example1.10}
Some important examples of quantum algebras are the following:
\begin{enumerate}
\item [(a)]\textit{Quantum algebra $\mathcal{U}'(\mathfrak{so}(3,K))$}.
It is the $K$-algebra generated by $I_1,I_2,I_3$ subject to relations
\[
I_2I_1-qI_1I_2=-q^{1/2}I_3,\ \ \ I_3I_1-q^{-1}I_1I_3=q^{-1/2}I_2,\ \ \
I_3I_2-qI_2I_3=-q^{1/2}I_1,
\]
where $q\in K-\{0\}$. Then, $\mathcal{U}'(\mathfrak{so}(3,K))\cong \sigma(K)\langle I_1, I_2, I_3
\rangle$ and it is known that
$Z(\cU'(\mathfrak{so}(3,K)))=K[-q^{1/2}(q^{2}-1)I_{1}I_{2}I_{3}+q^{2}I_{1}^{2}+I_{2}^{2}+q^{2}I^{2}_{3}]$
for $ q $ generic, and if $ q $ is a root of unity,
$Z(\cU'(\mathfrak{so}(3,K)))=K[-q^{1/2}(q^{2}-1)I_{1}I_{2}I_{3}+q^{2}I_{1}^{2}+I_{2}^{2}+q^{2}I^{2}_{3},C_n^1,C_n^2,C_n^3]$,
where $ C_n^k=2T_p(I_k(q-q^{-1})/2) $ for $ k=1,2,3 $ and $ T_p(x) $ is a Chebyshev polynomial of
the first kind (\cite{Levandovskyy}).
\item[(b)] \label{qeasl2k} \textit{Quantum enveloping algebra of $\mathfrak{sl}(2,K)$.}$ \quad $ $\mathcal{U}_q(\mathfrak{sl}(2,K))$ is defined
as the algebra generated by $x,y,z,z^{-1}$ with relations
\begin{align*}
zz^{-1}= z^{-1}z=1, \ \ xz= q^{-2}zx, yz= q^{2}zy, \ \ xy-yx= \frac{z-z^{-1}}{q-q^{-1}},
\end{align*}
where $q\neq 1,-1$. The above relations show that $\mathcal{U}_q(\mathfrak{sl}(2,K))=\sigma(K[z,z^{-1}])\langle
    x,y\rangle$. Then for $ q $ generic, its center is generated by the Casimir element $ C_2=(q^2-1)^2xy+qz+q^3z^{-1} $, and
    if $ q^n=1 $ for some $ n\geq 2 $, then $ Z(\mathfrak{sl}(2,K))=K[C_2,x^n,y^n,z^n,z^{-n}] $ (\cite{Levandovskyy}).
\item[(c)] \textit{Quantum symplectic space} $\cO_q(\mathfrak{sp}(K^{2n}))$.  It is the algebra generated by $ x_1,\dots,$ $x_{2n} $, satisfying the following relations:
\begin{align*}
x_jx_i=qx_ix_j,\quad i<j,\quad \ i'\neq j, &&
x_{i'}x_i=q^2x_ix_{i'}+(q^2-1)\sum_{1\leq k<i}q^{i-k}x_kx_{k'},\quad i<i',
\end{align*}
where $ i'=2n-i+1 $ for $ 1\leq i\leq 2n $. From the relations above we have that $
\cO_q(\mathfrak{sp}(K^{2n})) $ is isomorphic to a bijective skew $ PBW $ extension, $$
\cO_q(\mathfrak{sp}(K^{2n}))\cong \sigma(\sigma(\cdots \sigma(\sigma(\KK)\langle
x_1,x_{2n}\rangle)\langle x_2,x_{2n-1}\rangle)\cdots)\langle x_n,x_{n+1}\rangle .$$ In
\cite{Hechun} Zhang studied the irreducible representations of $ \cO_q(\mathfrak{sp}(K^{2n})) $,
where the main step is the computation of the center, thus in Theorem 4.3 therein is proved that if
$ m $ is an odd positive integer and if $ q $ is a primitive $ m- $th root of unity, then the
center is generated by $ x_i^m $ for all $ i=1,2,\dots,2n $. However, by Lemma 4.2 in
\cite{Hechun}, $ x_i^m $ are central elements if $ q $ is a $ m $th root of unity, and it is
possible to show from their relations that if $ q $ is generic then $
Z(\cO_q(\mathfrak{sp}(K^{2n})))=K $.
\item[(d)] \textit{Quantum Weyl algebra of Maltsiniotis} $ A_{n}^{q,\lambda} $. Let
$\textbf{q}=[q_{ij}]$ be a matrix over $K$ such that $q_{ij}q_{ji}=1$ and $q_{ii}=1$ for all $1\le
i,j\le n$. Fix an element $\lambda:=(\lambda_1,\dotsc, \lambda_n)$ of $(K^*)^n$. By definition,
this algebra is generated by $x_i,y_j,\ 1\le i,j\le n$ subject to the following relations: For any
$1\le i < j\le n$, $x_ix_j=\lambda_iq_{ij}x_jx_i$,  $y_iy_j=q_{ij}y_jy_i$, $x_iy_j=q_{ji}y_jx_i$, $
y_ix_j=\lambda_i^{-1}q_{ji}x_jy_i$. And for any $1\le i\le n$, $x_iy_i-\lambda_iy_ix_i=1+\sum_{1\le
j< i}(\lambda_j-1)y_jx_j$. From these relations we have that $A_n^{\textbf{q},\lambda}$ is
isomorphic (see \cite{Oswaldo}) to a bijective skew $PBW$ extension
\begin{center}
    $A_n^{\textbf{q},\lambda}\cong \sigma(\sigma(\cdots \sigma(\sigma(K)\langle
        x_1,y_1\rangle)\langle x_2,y_2\rangle)\cdots)\langle x_{n-1},y_{n-1}\rangle)\langle
    x_n,y_n\rangle$.
\end{center}
If none of $ \lambda_i $ is a root of unity, the quantum Weyl algebra of Maltisiniotis has trivial
center $ K $ (\cite{Tang}, Theorem 1.4).
\item[(e)] \textit{Multiparameter quantized Weyl algebra} $A_n^{Q,\Gamma}(K)$. Let $Q:=[q_1,\dotsc,q_n]$ be a vector in $K^{n}$
with no zero components, and let $\Gamma=[\gamma_{ij}]$ be a multiplicatively antisymmetric
$n\times n$ matrix over $K$. The multiparameter Weyl algebra $A_n^{Q,\Gamma}(K)$ is defined to be
the algebra generated by $K$ and the indeterminates $y_1,\dotsc,y_n,x_1,\dotsc,x_n$ subject to the
relations: $y_iy_j=\gamma_{ij}y_jy_i$,  $1\le i,j\le n$, $x_ix_j=q_i\gamma_{ij}x_jx_i$, $x_iy_j=\gamma_{ji}y_jx_i$, $ 1\le i<j\le n$,
 $x_iy_j=q_j\gamma_{ji}y_jx_i$, $ 1\le j<i\le n$, $x_jy_j=q_jy_jx_j +1 +\sum_{l<j}(q_l-1)y_lx_l$, $ 1\le j\le n$.
From these we have that $A_n^{Q,\Gamma}(K)$ is isomorphic to a bijective skew $PBW$ extension,
\begin{center}
    $A_n^{Q,\Gamma}(K)\cong \sigma(\sigma(\cdots \sigma(\sigma(K)\langle
        x_1,y_1\rangle)\langle x_2,y_2\rangle)\cdots)\langle x_{n-1},y_{n-1}\rangle)\langle
    x_n,y_n\rangle$.
\end{center}
In $ \cite{Tang} $ was studied the multiparameter quantized Weyl algebra of ``symmetric type",
where the relations are:
\begin{align*}
y_iy_j=&\gamma_{ij}y_jy_i, && x_ix_j=\gamma_{ij}x_jx_i, && 1\leq i<j\leq n,\\
x_iy_j=&\gamma_{ji}y_jx_i, && x_iy_j=\gamma_{ji}y_jx_i, && 1\le j<i\le n,\\
x_jy_j=&q_jy_jx_j +1 && 1\leq j\leq n.
\end{align*}
If none of $ q_i $ is a root of unity, then this algebra has a trivial center $ K $ (\cite{Tang},
Corollary 1.1).
\item[(f)] $ n- $\textit{multiparametric skew quantum space} over $ R $.
Let $ R $ be a ring with a fixed matrix of parameters $  \textbf{q} := [q_{ij}]\in M_{n}(R)$, $ n\geq 2 $, such that
$ q_{ii}=1=q_{ij}q_{ji} $ for every $ 1\leq i, j\leq n $, and suppose also that it is given a system $ \sigma_{1}, \dots,\sigma_{n} $ of automorphisms of $ R $.
The quasi-commutative bijective skew $ PBW $ extension $ R_{q,\sigma}[x_{1},\dots,s_{n}] $ defined by
\[x_{j}x_{i}=q_{ij}x_{i}x_{j}, \quad x_{i}r=\sigma_{i}(r)x_{i}, \quad r \in R, \quad 1 \leq i,j \leq n,  \]
is called the $ n- $\textit{multiparametric skew quantum space} over $ R $. When all automorphisms
of the extension are trivial, i.e., $ \sigma_{i}=i_{R} $, $ 1\leq i\leq n $, it is called $ n-
$\textit{multiparametric quantum space} over $ R $. If $ R=K $ is a field, then $
K_{q,\sigma}[x_{1},\dots,x_{n}] $ is called the $ n- $\textit{multiparametric skew quantum space},
and the case particular case when $ n=2 $ it is called \textit{skew quantum plane}; for trivial automorphisms we have
the $ n- $\textit{multiparametric quantum space} and the \textit{quantum plane}\label{key} (see \cite{Oswaldo}). \\
Shirikov studied the center of the quantum plane; if $ q $ is not a root of unity, then its center
is $ K $, if $ q $ is a root of unity of the degree $ m $, $ m\in \NN $, then the center is the
subalgebra generated by $ x^m $, $ y^m $ (\cite{Shirikov}, Theorem 2.2). A more general case was
studied extensively in \cite{Zhangetal3} assuming that $K$ is a commutative domain and the
parameters are roots of the unity, so if $ q_{ij}=exp(2\pi \sqrt{-1}k_{ij}/d_{ij}) $, where $
d_{ij}:=o(q_{ij})<\infty $, $ |k_{ij}|<d_{ij} $, and $ mcd(k_{ij},d_{ij})=1 $, then if $
L_i:=lcm\{d_{ij}\mid j=1,\dots,n\} $ and $ P:=K[x_{1}^{L_{1}},\dots,x_{n}^{L_{n}}] $, the center $
Z(K_{q}[x_{1},\dots,x_{n}])=P $ if and only if $Z( K_{q}[x_{1},\dots,x_{n}]) $ is a polynomial ring
(\cite{Zhangetal3}, Lemma 4.3).
\end{enumerate}
\end{example}
\begin{example}
The \textit{Jordan plane} is the $ K- $algebra defined by $ \mathcal{J}=K\{x,y\}/\langle
yx-xy-x^2\rangle $; it is a skew $ PBW $ extension of $ K[x] $. In fact, $ \cJ=\sigma(K[x])\langle
y\rangle $, with multiplication given by $ yx=xy-x^2 $. Then $ Z(\cJ)=K $ (if $ char(K)=p>0 $, then
$ Z(\cJ) $ is the subalgebra generated by $ x^p,y^p $, Theorem 2.2 in \cite{Shirikov}).
\end{example}

\section{New computations of centers of noncommutative algebras}

In this section we consider some other algebras that can be described as skew $PBW$ extensions and
for which the computation of the center probably has not been done before in the literature. We
classify these algebras according to the parameters that define the extension. Our computations
produce two groups: Algebras with trivial center and algebras with polynomial center.

We start with extensions of two variables. Let $ A $ be the $K$-algebra generated by $ x,y $
subject to relation
\begin{equation}\label{equation2.1}
 yx=q_1xy+q_2x+q_3y+q_4,
\end{equation}
with $ q_i\in K $, $ i=1,2,3,4 $ and $ q_1\neq 0 $. Note that $A$ is a skew $ PBW $ extension of $
K $ with trivial endomorphisms and derivations (see Remark \ref{remark1.7}).
\begin{lemma}
For  $ n\geq 2 $, in $A$ the following rules hold:
    \begin{enumerate}
        \item \begin{align*}
        yx^{n}=&\sum_{j=0}^{n}\binom{n}{j}q_{1}^{n-j}q_{3}^{j}x^{n-j}y+q_{2}\left(\sum_{i=0}^{n-1}q_{1}^{i}\right)x^{n}+
        \left[ q_{4}\left(\sum_{i=0}^{n-1}q_{1}^{i} \right)+q_{2}q_{3}\left(\sum_{i=0}^{n-2}\binom{i+1}{1}q_{1}^{i} \right)\right]x^{n-1}+\\
        &\cdots+ \left[q_{3}^{j-1}q_{4}\left(\sum_{i=0}^{n-j}\binom{i+j-1}{j-1}q_{1}^{i} \right)+q_{2}q_{3}^{j}\left(\sum_{i=0}^{n-j-1}\binom{i+j}{j}q_{1}^{i}
        \right) \right] x^{n-j}.
        \end{align*}
        \item \begin{align*}
        y^{n}x=&\sum_{j=0}^{n}\binom{n}{j}q_{1}^{n-j}q_{2}^{j}xy^{n-j}+q_{3}\left(\sum_{i=0}^{n-1}q_{i}^{i}\right)y^{n}+
        \left[q_{4}\left(\sum_{i=0}^{n-1}q_{1}^{i} \right)+q_{2}q_{3}\left( \sum_{i=0}^{n-2}\binom{n+1}{1}q_{1}^{i}\right)  \right]y^{n-1}+
        \\
        & \cdots+\left[q_{2}^{j-1}q_{4}\left(\sum_{i=0}^{n-j}\binom{i+j-1}{j-1}q_{1}^{i} \right)+q_{2}^{j}q_{3}\left(\sum_{i=0}^{n-j-1}\binom{i+j}{j}q_{1}^{i}
        \right) \right]y^{n-j}.
        \end{align*}
    \end{enumerate}
\end{lemma}
\begin{proof}
    We proceed by induction on $ n $. For $ n=2 $ the equalities hold trivially. Then we assume the formulae for $ n $
    and we will prove them for
    $ n+1 $. To simplify, we can write the product as $ yx^n=(b_{n}x^{n}+\dots+b_{1}x+b_{0})y+(a_{n}x^{n}+\dots+a_{1}x+a_{0}) $,
    where $b_{n-j}=\binom{n}{j}q_{1}^{n-j}q_{3}^{j}$ ,
    $a_{n}=q_{2}\sum_{i=0}^{n-1}q_{1}^{i}$ ,
    $a_{n-1}=q_{4}\left(\sum_{i=0}^{n-1}q_{1}^{i} \right)+q_{2}q_{3}\left(\sum_{i=0}^{n-2}\binom{i+1}{1}q_{1}^{i}\right)$, and
    $a_{n-j}= q_{3}^{j-1}q_{4}\left(\sum_{i=0}^{n-j}\binom{i+j-1}{j-1}q_{1}^{i} \right)+q_{2}q_{3}^{j}\left(\sum_{i=0}^{n-j-1}\binom{i+j}{j}q_{1}^{i} \right)$. Then,
    \begin{align*}
    yx^{n+1}=&(yx^{n})x\\
    =&(b_{n}x^{n}+\dots+b_{1}x+b_{0})yx+(a_{n}x^{n+1}+\dots+a_{1}x^{2}+a_{0}x)\\
    =&(b_{n}x^{n}+\dots+b_{1}x+b_{0})(q_{1}xy+q_{2}x+q_{3}y+q_{4})+(a_{n}x^{n+1}+\dots+a_{1}x^{2}+a_{0}x)\\
    =&\left[b_{n}q_{1}x^{n+1}+(b_{n-1}q_{1}+b_{n}q_{3})x^{n}+\dots+(b_{n-j}q_{1}+b_{n+1-j}q_{3})x^{n+1-j}+\dots+b_{0}q_{3} \right]y+ (a_{n}+b_{n}q_{2})\\&x^{n+1}+(a_{n-1}+b_{n-1}q_{2}+b_{n}q_{4})x^{n}+\dots+
    (a_{n-j}+b_{n-j}q_{2}+b_{n+1-j}q_{4})x^{n+1-j}+\dots+b_{0}q_{4}.
    \end{align*}
    Now we compare coefficients: For $ x^{n+1-j}y $, $b_{n-j}q_{1}+b_{n+1-j}q_{3}=\binom{n}{j}q_{1}^{n-j}q_{3}^{j}q_{1}+\binom{n}{j-1}q_{1}^{n-(j-1)}q_{3}^{j-1}q_{3}
    = \left[\binom{n}{j}+\binom{n}{j-1} \right]q_{1}^{n+1-j}q_{3}^{j}
    =\binom{n+1}{j}q_{1}^{n+1-j}q^{j}_{3}$ is the coefficient of $ x^{n+1-j}y  $ in the expansion of $ yx^{n+1} $. For  $ x^{n+1-j} $ we have: $a_{n-j}+b_{n-j}q_{2}+b_{n+1-j}q_{4}=q_{3}^{j-1}q_{4}\left(\sum_{i=0}^{n-j}\binom{i+j-1}{j-1}q_{1}^{i} \right)+q_{2}q_{3}^{j}\left(\sum_{i=0}^{n-j-1}\binom{i+j}{j}q_{1}^{i} \right)
    +\binom{n}{j}q_{1}^{n-j}q_{3}^{j}q_{2}+\binom{n}{j-1}q_{1}^{n-(j-1)}q_{3}^{j-1}q_{4}
    =q_{3}^{j-1}q_{4}\left(\sum_{i=0}^{n+1-j}\binom{i+j-1}{j-1}q_{1}^{i} \right)+q_{2}q_{3}^{j}\left(\sum_{i=0}^{n-j}\binom{i+j}{j}q_{1}^{i} \right)$. Similarly for  $ x^{n+1} $ y $ x^{n} $. So,
    \begin{align*}
    yx^{n+1}=&\sum_{j=0}^{n+1}\binom{n+1}{j}q_{1}^{n+1-j}q_{3}^{j}x^{n+1-j}y+q_{2}(\sum_{i=0}^{n}q_{1}^{i})x^{n+1}+
    \left[ q_{4}\left(\sum_{i=0}^{n}q_{1}^{i} \right)+q_{2}q_{3}\left(\sum_{i=0}^{n-1}\binom{i+1}{1}q_{1}^{i} \right)\right]x^{n}+\\
    &\dots+\left[q_{3}^{j-1}q_{4}\left(\sum_{i=0}^{n+1-j}\binom{i+j-1}{j-1}q_{1}^{i} \right)+q_{2}q_{3}^{j}\left(\sum_{i=0}^{n-j}\binom{i+j}{j}q_{1}^{i} \right) \right] x^{n+1-j}.
    \end{align*}
    This completes the proof of the first equality. The second equality can be proved similarly.
\end{proof}
In particular, if $ q_{2}=q_{3}=0 $, then
\begin{align*}
yx^{n}=&q_{1}^{n}x^{n}y+q_{4}\left(\sum_{i=0}^{n-1}q_{1}^{i} \right)x^{n-1},\\
y^{n}x=&q_{1}^{n}xy^{n}+q_{4}\left(\sum_{i=0}^{n-1}q_{1}^{i} \right)y^{n-1}.
\end{align*}

\begin{theorem}
Let $ A=\sigma(K)\langle x,y\rangle $ be the skew $ PBW $ extension of $K$ defined as in
(\ref{equation2.1}).
    \begin{enumerate}
        \item[\rm (i)] If $ q_{2}=q_{3}=0 $, then \\
        $ Z(A)=
        \begin{cases}
        K &\text{ if } q_{1} \text{ is not a root of unity, or, } q_{1}=1 \text{ and } q_{4}\neq 0;\\
        K[x^{n},y^{n}] & \text{ if } q_{1} \text{ is a root of unity of degree  } n\ \geq 2.
        \end{cases}$\\
        \item[\rm (ii)] If $ q_{2}\neq 0 $, or, $ q_{3}\neq 0 $, then\\
        $ Z(A)=
        \begin{cases}
        K &\text{ if } q_{1} \text{ is not a root of unity, or, } q_{1}=1;\\
        K[f(x),g(y)] & \text{ if } q_{1} \text{ is a root of unity }
        \text{of degree } n \ \  \geq 2,\\ &\text{ with } deg(f)=deg(g)=n.
        \end{cases}$\\
    \end{enumerate}
\end{theorem}
\begin{proof}
    (i) It is clear that $ K\subseteq Z(A) $;  let $ f=\sum_{i=1}^{t} c_{i}x^{\alpha_{i1}}y^{\alpha_{i2}}\in Z(A) $, with $ c_{i}\in K $. Since $ xf=fx $, then
    $\sum_{i=1}^{t} c_{i}x^{\alpha_{i1}+1}y^{\alpha_{i2}}
    =\sum_{i=1}^{t} c_{i}x^{\alpha_{i1}}y^{\alpha_{i2}}x
    =\sum_{i=1}^{t} c_{i}q_{1}^{\alpha_{i2}}x^{\alpha_{i1}+1}y^{\alpha_{i2}}+\sum_{i=1}^{t}
    c_{i}q_{4}\left(\sum_{j=0}^{\alpha_{i2}-1}q_{1}^{j}\right)x^{\alpha_{i1}}y^{\alpha_{i2}-1}$.

    From this we get that $ c_{i}=c_{i}q_{1}^{\alpha_{i2}} $ and $ \sum_{i=1}^{t} c_{i}q_{4}\left(\sum_{j=0}^{\alpha_{i2}-1}q_{1}^{j}
    \right)x^{\alpha_{i1}}y^{\alpha_{i2}-1}=0 $, so:
    \begin{enumerate}
        \item  If $ q_{1} $ is not a root of unity, then $ \alpha_{i2}=0 $, thus $ f=\sum_{i=1}^{t} c_{i}x^{\alpha_{i1}} $, but $ fy=yf $, then \begin{align*}
        \sum_{i=1}^{t} c_{i}x^{\alpha_{i1}}y=&\sum_{i=1}^{t} c_{i}q_{1}^{\alpha_{i1}}x^{\alpha_{i1}}y+\sum_{i=1}^{t} c_{i}q_{4}
        \left(\sum_{k=0}^{\alpha_{i1}-1}q_{1}^{k} \right)x^{k-1},\\
        \end{align*}
        whence $ c_{i}=c_{i}q_{1}^{\alpha_{i1}} $, and this is valid if $ i=0 $, so $ f=c_{0}\in K $. Therefore, $ Z(A)\subseteq K\subseteq Z(A) $.
        \item If $ q_{1} $ is a root of unity of degree $ n \geq 2$, $ c_{i}=c_{i}q_{1}^{\alpha_{i2}} $ implies that  $ \alpha_{i2}=nl $
        with $ l\in \NN $, also $\sum_{i=0}^{ln-1}q_{1}^{i}=0$, so $ f=\sum_{i=1,l\geq 0}^{t} c_{i,l}x^{i}y^{nl} $. But $ fy=yf $,
        then

        \begin{align*}
        \sum_{i=1}^{t} c_{il}x^{\alpha_{i1}}y^{nl+1}=&\sum_{i=1}^{t} c_{il}q_{1}^{i}x^{\alpha_{i1}}y^{nl+1}+\sum_{i=0}^{t}
        c_{il}q_{4}\left(\sum_{k=0}^{\alpha_{i1}-1}q_{1}^{k} \right)x^{k-1}y^{nl}.\\
        \end{align*}
        Now, $ c_{il}=c_{il}q_{1}^{i} $, thus  $ i=nr $, with $ r\in \mathbb{N} $, in this case $\sum_{i=0}^{nr-1}q_{1}^{i}=0$,
        and hence $ f=\sum c_{rl}(x^{n})^{r}(y^{n})^{l} $, whence $ K[x^{n},y^{n}] \subseteq Z(A)$.
        Finally, since $ x^{n},y^{n} $ are central elements of $ A $, then $ K[x^{n},y^{n}] = Z(A)$.
    \end{enumerate}
(ii) From the basic conditions of commutation $xf=fx, yf=fy$, the proof of this part is similar and
it is possible to find the desired polynomials $f(x)$ and $g(y)$.
\end{proof}

\begin{remark}
    \begin{enumerate}
\item The precise description of polynomials $f(x)$ and $g(y)$ in the previous theorem it is in general
difficult. However, in some particular cases the computation is possible: If $ q_{1} $ is a root of
unity of degree $ n $, $ q_{3}=0 $ and $ q_{2}\neq 0 $, then $ f(x)=x^{n} $, and if $ q_{3}\neq 0 $
and $ q_{2}=0 $, then $ g(y)=y^{n} $.
        \item Some known results can be obtained as particular cases of the previous theorem:
        \begin{enumerate}
        \item If $ q_{1}=q_{4}=1 $ and $ q_{2}=q_{3}=0 $, then $ A $ is the Weyl algebra and $Z(A)=K$.
        \item If $ q_{2}=q_{3}=0 $ then $ A$ is the $ q_{1}$-quantum Weyl
        algebra; if $ q_{1} $ is a primitive $ n- $th roof of unity for some $ n\geq 2 $, then $ Z(A)=K[x^{n},y^{n}]
        $. In \cite{Zhangetal3}, section 2, it is proved that $ A $ is free over $ Z(A)$ of rank $ n^{2} $, with basis
        $ \mathcal{B}=\{x^{i}y^{j}\mid 0\leq i,j\leq n-1 \} $.
        \item If $ L $ is a two-dimensional solvable algebra generated by $ x,y $, with $ [x,y]=x $, then for its universal enveloping algebra we get $Z(\mathcal{U}(L))=K$. This particular
        case was considered in \cite{Makar}.
        \end{enumerate}
    \end{enumerate}
\end{remark}

Now we pass to study the center of extensions with $n\geq 3$ variables. Related to Example
\ref{Example1.10} (f), we have the following easy result.

\begin{lemma}[$ n- $\textit{multiparametric quantum space}]\label{quasi}
    For $ q_{ij}=\frac{l_{ij}}{k_{ij}}\in K- \{0\} $, $ (l_{ij},k_{ij})=1 $, $ 1\leq i,j\leq n $ ($ q_{ii}=1 $), let $ A $ be the
    $K$-algebra defined as the quasi-commutative skew $ PBW $ extension given by $ x_{j}x_{i}=q_{ij}x_{i}x_{j} $, for all $ 1\leq i<j\leq n $.
    If  $ q_{ij} $ is not a root of unity and $ (l_{ij},l_{rs})=(k_{ij},k_{rs})=(l_{ij},k_{rs})=1 $ for $ i,j,r,s\in\{1,2,\dots,n\} $, then $ Z(A)=K $.
\end{lemma}
\begin{proof}
    In this case we have the rule:
    \begin{center}
    $ x_{k}^{l_{k}}x_{j}^{l_{j}}x_{i}^{l_{i}}=(q_{ij}^{l_{i}l_{j}}q_{ik}^{l_{i}l_{k}}q_{jk}^{l_{j}l_{k}})x_{i}^{l_{i}}x_{j}^{l_{j}}x_{k}^{l_{k}} $.
    \end{center}
    Let $ f=\sum_{i=0}^{l}c_{i}x_{1}^{\alpha_{i1}}\dotsc x_{n}^{\alpha_{in}}\in Z(A) $, with $ \alpha_{ij}\geq 0 $ and $ c_{i}\in K-\{0\} $, for $ i=0,1,\dots,l $, and $ j=1,2,\dots,n $. Since $ x_{1}f=fx_{1} $, $ \dots $, $ x_{n}f=fx_{n} $, then
    \begin{align*}
    \sum_{i=0}^{l}c_{i}x_{1}^{\alpha_{i1}+1}\dotsc x_{n}^{\alpha_{in}}=&\sum_{i=0}^{l}c_{i}q_{1n}^{\alpha_{in}}\dotsc q_{12}^{\alpha_{i2}}x_{1}^{\alpha_{i1}+1}\dotsc x_{n}^{\alpha_{in}} \\
    \vdots\quad=&\quad\vdots\\
    \sum_{i=0}^{l}c_{i}q_{n-1,n}^{\alpha_{i(n-1)}}\dotsc q_{1n}^{\alpha_{i1}}x_{1}^{\alpha_{i1}}\dotsc x_{n}^{\alpha_{in}}=
    &\sum_{i=0}^{l}c_{i}x_{1}^{\alpha_{i1}}\dotsc x_{n}^{\alpha_{in}+1}.
    \end{align*}
        From this we get that
        $ c_{i}=c_{i}q_{1n}^{\alpha_{in}}\dotsc q_{12}^{\alpha_{i2}} $, $\cdots$, $ c_{i}=c_{i}q_{1n}^{\alpha_{i1}}
        \cdots q_{n-1,n}^{\alpha_{i(n-1)}} $,  so $ 1=q_{1n}^{\alpha_{in}}\dotsc q_{12}^{\alpha_{i2}} $,$\dots$, $ 1=q_{1n}^{\alpha_{in}}\dotsc
        q_{12}^{\alpha_{i2}} $  for $ i=0,1,2,\dots,l $. Since $ q_{ij} $ is not a root of unity and by hypothesis, then necessarily $ \alpha_{ij} =0$ for $ i=0,1,\dots,l $,
        and $ j=1,\dots,n $. Thus, $ f=c_{0}\in K $. Therefore, $ Z(A)\subseteq K \subseteq Z(A) $.
\end{proof}
\begin{theorem}
    Let $ A=\sigma(K)\langle x_1,\dots,x_n\rangle $ be a skew $ PBW $ extensions of $ K $ defined by
\begin{center}
$ x_jx_i=q_{ij}x_ix_j+a_{ij}^{(1)}x_1+\cdots +a_{ij}^{(n)}x_n+d_{ij}$,
\end{center}
with $ 1\leq i<j\leq n$, and $ q_{ij}=\frac{l_{ij}}{k_{ij}} $, $ (l_{ij},l_{rs})=(k_{ij},k_{rs})=(l_{ij},k_{rs})=1 $ for $ i,j,r,s\in\{1,2,\dots,n\} $. If $ q_{ij} $ is not a root of unity, then $ Z(A)=K $.
\end{theorem}
\begin{proof}
According to Theorem \ref{filteredskew}, $ Gr(A) $ is a quasi-commutative skew $ PBW $ extension
with relations $ x_jx_i=q_{ij}x_ix_j $, $ 1\leq i<j\leq n$ (since $A$ is a $K$-algebra,
$\sigma_i=i_R$ and $\delta_i=0$, see Remark \ref{remark1.7}). Then, by the previous lemma, $
K=Z(Gr(A))\supseteq Gr(Z(A)) $. Thus $ Gr(Z(A))=K $ and therefore $ Z(A)=K $.
 \end{proof}

The following theorem allows us to give a wide list of algebras which center probably has not been
computed before.
\begin{theorem}\label{operator}
    Let $ A=\sigma(K[x_{1},\dots,x_{n}])\langle y_{1},\dots,y_{n}\rangle $ be the skew $PBW$ extension defined by
    \begin{align}
    y_{i}y_{j}&=y_{j}y_{i}\quad 1\leq i<j\leq n,\\
    y_{j}x_{i}&=x_{i}y_{j}\quad i\neq j,\\
    y_{i}x_{i}&=q_{i}x_{i}y_{i}+d_{i}y_{i}+a_{i}, \quad 1\leq i\leq n. \label{Def2.4}
    \end{align}
     Then,
    \begin{enumerate}
\item $ Z(A)=K $ in any of the following cases:
    \begin{enumerate}
        \item For $1\leq i\leq n$, $ q_{i} $ is not a root of unity.
        \item If for some $i$, $ q_{i}=1 $, then $ d_{i}\neq 0$ or $a_{i}\neq 0$.
    \end{enumerate}
    \item If for every $ 1\leq i\leq n $, $ q_{i} $ is a primitive $ l_{i}- $th root of unity of degree $ l_{i}\geq 2 $,
    then the elements $ y_{i}^{l_{i}} $ are central and $ K\subsetneq Z(A) $.
    \end{enumerate}
\end{theorem}
\begin{proof}
\textit{Step 1}. For $ k=1,\dots,n $, in $A$ the following relations hold for every $m\geq 1$:
\begin{enumerate}
    \item[]
    $y_{k}x_{k}^{m}=\sum_{j=0}^{m}\binom{m}{j}q_{k}^{m-j}d_{k}^{j}x_{k}^{m-j}y_{k}+a_{k}\left(\sum_{i=0}^{m-1}q_{k}^{i}
    \right)x_{k}^{m-1}+\dots+d_{k}^{j-1}a_{k}\left(\sum_{i=0}^{m-j}\binom{i+j-1}{j-1}q_{k}^{i}\right)x_{k}^{m-j}$,
    \item[]
    $y_{k}^{m}x_{k}= q_{k}^{m}x_{k}y_{k}^{m}+d_{k}\left(\sum_{i=0}^{m-1}q_{k}^{i} \right)y_{k}^{m}+a_{k}\left(\sum_{i=0}^{m-1}q_{k}^{i} \right)y_{k}^{m-1}$.
\end{enumerate}
Hence, the elements $ y_k^{l_k} $ are central, assuming that $ q_k $ is a $ l_k- $th root of unity.

We verify the first equality by induction. The proof of the second identity is similar. The case $
m=1 $ is (\ref{Def2.4}). Then we assume it is valid for $ m $. So,
\begin{align*}
y_{k}x_{k}^{m+1}=&(y_{k}x_{k}^{m})x_{k}\\
=&(\sum_{j=0}^{m}\binom{m}{j}q_{k}^{m-j}d_{k}^{j}x_{k}^{m-j}y_{k}+a_{k}\left(\sum_{i=0}^{m-1}q_{k}^{i}\right)x_{k}^{m-1}+\dots+
d_{k}^{j-1}a_{k}\left(\sum_{i=0}^{m-j}\binom{i+j-1}{j-1}q_{k}^{i}\right)x_{k}^{m-j})x\\
=&\sum_{j=0}^{m}\binom{m}{j}q_{k}^{m-j}d_{k}^{j}x_{k}^{m-j}(q_{k}x_{k}y_{k}+d_{k}y_{k}+a_{k})+a_{k}\left(\sum_{i=0}^{m-1}q_{k}^{i}\right)x_{k}^{m}+\dots+\\
&d_{k}^{j-1}a_{k}\left(\sum_{i=0}^{m-j}\binom{i+j-1}{j-1}q_{k}^{i}\right)x_{k}^{m+1-j}.\\
=&\sum_{j=0}^{m}\binom{m}{j}q_{k}^{m+1-j}d_{k}^{j}x_{k}^{m+1-j}y_{k}+\sum_{j=0}^{m}\binom{m}{j}q_{k}^{m-j}d_{k}^{j+1}x_{k}^{m-j}y_{k}+a_{k}\sum_{j=0}^{m}\binom{m}{j}q_{k}^{m-j}d_{k}^{j}x_{k}^{m-j}y_{k}+\\ &a_{k}\left(\sum_{i=0}^{m-1}q_{k}^{i}\right)x_{k}^{m}+\dots+
d_{k}^{j-1}a_{k}\left(\sum_{i=0}^{m-j}\binom{i+j-1}{j-1}q_{k}^{i}\right)x_{k}^{m+1-j}.\\
=&\sum_{j=0}^{m+1}\binom{m+1}{j}q_{k}^{m+1-j}d_{k}^{j}x_{k}^{m+1-j}y_{k}+a_{k}\left(\sum_{i=0}^{m}q_{k}^{i}\right)x_{k}^{m}+\dots+\\ &d_{k}^{j-1}a_{k}\left(\sum_{i=0}^{m+1-j}\binom{i+j-1}{j-1}q_{k}^{i}\right)x_{k}^{m+1-j}.\\
=&y_{k}x_{k}^{m+1}.
\end{align*}

\textit{Step 2}. Let $ f=\sum_{i=1}^{t}P_{i}(X)y_{1}^{\alpha_{i1}}\cdots y_{n}^{\alpha_{in}}\in
Z(A) $, with $ \alpha_{ij}\geq 1 $   and $ P_{i}(X)\in K[x_{1},\dots,x_{n}]-\{0\} $ for $
i=1,\dots,t $ and $ j=1,2,\dots,n $. Since $ x_{1}f=fx_{1} $, then
    \begin{align*}
    \sum_{i=1}^{t}P_{i}(X)x_{1}y_{1}^{\alpha_{i1}}\cdots y_{n}^{\alpha_{in}}=&\sum_{i=1}^{t}P_{i}(X)y_{1}^{\alpha_{i1}}\cdots y_{n}^{\alpha_{in}}x_{1}\\
    =&\sum_{i=1}^{t}P_{i}(X)y_{1}^{\alpha_{i1}}x_{1}\cdots y_{n}^{\alpha_{in}}\\
    =&\sum_{i=1}^{t}P_{i}(X)( q_{1}^{\alpha_{i1}}x_{1}y_{1}^{\alpha_{i1}}+d_{1}\left(\sum_{i=0}^{\alpha_{i1}-1}q_{1}^{i} \right)y_{1}^{\alpha_{i1}}+\\
    &a_{1}\left(\sum_{i=0}^{\alpha_{i1}-1}q_{1}^{i} \right)y_{1}^{\alpha_{i1}-1})y_{2}^{\alpha_{i2}}\dots y_{n}^{\alpha_{in}}\\
    =&\sum_{i=1}^{t}q_{1}^{\alpha_{i1}}P_{i}(X) x_{1}y_{1}^{\alpha_{i1}}\cdots y_{n}^{\alpha_{in}}+\sum_{i=1}^{t}P_{i}(X)(d_{1}\left(\sum_{i=0}^{\alpha_{i1}-1}q_{1}^{i} \right)y_{1}^{\alpha_{i1}}\\
    &+
    a_{1}\left(\sum_{i=0}^{\alpha_{i1}-1}q_{1}^{i} \right)y_{1}^{\alpha_{i1}-1})y_{2}^{\alpha_{i2}}\dots y_{n}^{\alpha_{in}}.\\
    \end{align*}
    Since $ x_{1}f-fx_{1}=0 $ and $ A $ is a left $ K[x_{1},\dots,x_{n}]- $free  module with canonical basis $ \{y_{1}^{\alpha_{i1}}y_{2}^{\alpha_{i2}}\dots
    y_{n}^{\alpha_{in}}\mid (\alpha_{i1},\dots,\alpha_{in})\in \NN^{n} \} $, then $ P_{i}(X)=q_{1}^{\alpha_{i1}}P_{i}(X) $, and
     $ \sum_{i=1}^{t}d_{1}P_{i}(X)\left(\sum_{i=0}^{\alpha_{i1}-1}q_{1}^{i} \right)y_{1}^{\alpha_{i1}}y_{2}^{\alpha_{i2}}\dots y_{n}^{\alpha_{in}}+$ \\
     $\sum_{i=1}^{t}a_{1}\left(\sum_{i=0}^{\alpha_{i1}-1}q_{1}^{i} \right)y_{1}^{\alpha_{i1}-1}y_{2}^{\alpha_{i2}}\dots y_{n}^{\alpha_{in}}=0 $.
     Thus, if $ q_{1} $ is not root of unity, then
    the first equality implies that $ \alpha_{i1}=0 $ for $ 1\leq i\leq t $. If $ q_{1}=1 $,
    then we compare the leader monomial of \\$ \sum_{i=1}^{t}d_{1}P_{i}(X)\left(\sum_{i=0}^{\alpha_{i1}-1}q_{1}^{i}
    \right)y_{1}^{\alpha_{i1}}y_{2}^{\alpha_{i2}}\dots y_{n}^{\alpha_{in}}$ if $ d_{1}\neq 0 $, or $ \sum_{i=1}^{t}a_{1}\left(\sum_{i=0}^{\alpha_{i1}-1}q_{1}^{i}
    \right)y_{1}^{\alpha_{i1}-1}y_{2}^{\alpha_{i2}}\dots y_{n}^{\alpha_{in}} $ otherwise. Since $ A $ is a domain and $ char(K)=0 $,
    then necessarily $ \alpha_{i1}=0 $ for $ 1\leq i\leq t $,
    i.e., $ f  $ does not have powers of $ y_{1} $. Repeating this procedure with $ x_{2},\dots,x_{n} $ we conclude that $ f=P(X)=\sum_{i=1}^{r}k_{i}x_{1}^{\beta_{i1}}\dots x_{n}^{\beta_{in}}\in K[x_{1},\dots,x_{n}]-\{0\} $, but  $ y_{1}f=fy_{1} $, so
    \begin{align*}
    \sum_{i=1}^{r}k_{i}x_{1}^{\beta_{i1}}\dots x_{n}^{\beta_{in}}y_{1}=&y_{1}\sum_{i=1}^{r}k_{i}x_{1}^{\beta_{i1}}\dots x_{n}^{\beta_{in}}\\
    =&\sum_{i=1}^{r}k_{i}(\sum_{j=0}^{\beta_{i1}}\binom{\beta_{i1}}{j}q_{1}^{\beta_{i1}-j}d_{1}^{j}x_{1}^{\beta_{i1}-j}y_{1}+a_{1}\left(\sum_{j=0}^{\beta_{i1}-1}q_{1}^{j}\right)x_{1}^{\beta_{i1}-1}+\dots+d_{1}^{j-1}a_{1}\\
    &\left(\sum_{l=0}^{\beta_{i1}-j}\binom{l+j-1}{j-1}q_{1}^{l}\right)x_{1}^{\beta_{i1}-l})x_{2}^{\beta_{i2}}\cdots
    x_{n}^{\beta_{in}};
    \end{align*}
    then comparing terms, $ k_{i}=k_{i}q_{1}^{\beta_{i1}} $ implies $ \beta_{i1}=0 $ if $ q_{1} $ is not root of unity.
    If $ q_{1}=1 $ and $ d_{1}\neq 0 $ then $ k_{i}\binom{\beta_{i1}}{j}q_{1}^{\beta_{i1}-j}d_{1}^{j}=0 $, or $ a_{1}\sum_{j=0}^{\beta_{i1}-1}q_{1}^{i}=0 $ if
    $ d_{1}=0 $; in any case, it is only possible if $ \beta_{i1}=0 $ , $ 1\leq i\leq r $, i.e., $ f=\sum_{i=1}^{r}k_{1}x_{2}^{\beta_{i2}}\cdots x_{n}^{\beta_{in}} $.
    Performing the same procedure with $ y_{2},\dots, y_{n} $, we conclude that $ f=k\in K $. Therefore $ K\subseteq Z(A)\subseteq K $.\\
    In the case when each $ q_{i} $ is a root of unity of degree $ l_{i} $, then $ q_{i}^{l_{i}}=1 $ and $ \sum_{j=0}^{l_{i}-1}q_{i}^{j}=0 $.
    Thus     the relation $ y_{i}^{l_{i}}x_{i}= q_{i}^{l_{i}}x_{i}y_{i}^{l_{i}}+d_{i}\left(\sum_{j=0}^{l_{i}-1}q_{i}^{j}
    \right)y_{i}^{l_{i}}+a_{i}\left(\sum_{j=0}^{l_i-1}q_{i}^{j} \right)y_{i}^{l_{i}-1}=x_{i}y_{i}^{l_{i}}$, and since $ y_{i}y_{j}=y_{j}y_{i} $,
    $ y_{i}x_{j}=x_{j}y_{i} $, then for $ 1\leq i\leq n $, $ y_{i}^{l_{i}} $ is a central element of $ A $. Therefore, $ K\subsetneq Z(A) $.
\end{proof}

\begin{remark} In (\ref{Def2.4}), if $ d_k=0 $, then $
y_{k}x_{k}^{m}=q_{k}^{m}x_{k}^{m}y_{k}+a_{k}\left(\sum_{i=0}^{m-1}q_{k}^{i}\right)x_{k}^{m-1} $,
and if $ q_i $ is a root of unity of degree $ l_i $ for $ 1\leq i\leq n $, then $ x_i^{l_i} $ are
central elements. Thus, $ Z(A)=K[x_1^{l_1},\dots,x_n^{l_n},y_1^{l_1},\dots,y_n^{l_n}] $. Finally,
if some $ q_i $ is not a root of unity, then we exclude the variables $ x_i $, $ y_i $ from the
previous center.
\end{remark}

\begin{examples}
    According to Theorem \ref{operator}, the following algebras have trivial center:
    \begin{enumerate}
        \item \textit{Operator algebras}.\label{operatoralgebras} Some important well-known operator
        algebras can be described as skew $ PBW $ extensions of $ K[t_1,\dots,t_n]$.
        \begin{enumerate}
            \item [\rm (a)]\textit{Algebra of linear partial differential operators.}
            The $n$-th Weyl algebra $A_n(K)$ over $K$ coincides with the
            $K$-algebra of linear partial differential operators with
            polynomial coefficients $K[t_1,\dots,t_n]$.
            The generators of $A_n(K)$ satisfy the following
            relations $t_it_j=t_jt_i$, $ \partial_i\partial_j=\partial_j\partial_i$, $ 1\le i<j\le n$,
            $\partial_jt_i=t_i\partial_j+\delta_{ij}$, $ 1\le i,j\le n$, where $\delta_{ij}$ is the Kronecker symbol.
            Thus, $ q_{i}=1 $, $ d_{i}=0 $ and $ a_{i}=1 $ for $ 1\leq i\leq n $.
            \item [\rm (b)]\textit{Algebra of linear partial shift operators.} The $K$-algebra of linear partial \textit{shift}
            (\textit{recurrence}) operators with polynomial, respectively with rational coefficients, is
            $K[t_1,\dotsc,t_n][E_1,\dotsc,E_n]$, respectively
            $K(t_1,\dotsc,t_n)[E_1,\dotsc,E_n]$,
            subject to the relations: $ t_jt_i=t_it_j $, $ E_jE_i=E_iE_j $ for $ 1\le i<j\le n$, and $ E_it_i=(t_i+1)E_i=t_iE_i+E_i$,
            $1\le i\le n$, $E_jt_i=t_iE_j$, $ i\neq j$. So $ q_{i}=1 $, $ d_{i}=1 $ and $ a_{i}=0 $ for $ 1\leq i\leq n$.
            \item [\rm (c)]\textit{Algebra of linear partial difference operators.} The $K$-algebra of
            linear partial \textit{difference opertors} with polynomial,
            respectively with rational coefficients, is
            $K[t_1,\dotsc,t_n][\Delta_1,\dotsc,\Delta_n]$,
            respectively $K(t_1,\dotsc,t_n)[\Delta_1,\dotsc,\Delta_n]$, subject to the relations:

            $t_jt_i=t_it_j$, $1\le i<j\le n$,
            $\Delta_it_i=(t_i+1)\Delta_i+1=t_i\Delta_i+\Delta_i+1$, $ 1\le i\le n$,
            $\Delta_jt_i=t_i\Delta_j$, $i\neq j$,
            $\Delta_j\Delta_i=\Delta_i\Delta_j$, $1\le i<j\le n$. Thus $ q_{i}=1 $, $ d_{i}=1 $ and $ a_{i}=1 $ for $ 1\leq i\leq n$.
            \item [\rm (d)]\textit{Algebra of linear partial $q$-dilation operators.} For a fixed $q\in K-\{0\}$, the
            $K$-algebra of linear partial $q$-\textit{dilation operators} with
            polynomial coefficients, respectively, with rational coefficients,
            is $K[t_1,\dotsc,t_n][H_1^{(q)},\dotsc,H_n^{(q)}]$,
            respectively
            $K(t_1,\dotsc,t_n)[H_1^{(q)},\dotsc,H_n^{(q)}]$,
            subject to the relations:
            $t_jt_i=t_it_j$, $1\le i<j\le n$,
            $H_i^{(q)}t_i=qt_iH_i^{(q)}$, $1\le i\le n$,
            $H_j^{(q)}t_i=t_iH_j^{(q)}$, $i\neq j$,
            $H_j^{(q)}H_i^{(q)}=H_i^{(q)}H_j^{(q)}$, $1\le i<j\le n$.
            Thus $ q_{i}=q $, $ d_{i}=0 $ and $ a_{i}=0 $ for $ 1\leq i\leq n $.
            \item [\rm (e)]\textit{Algebra of linear partial $q$-differential operators.} For a fixed $q\in K-\{0\}$, the
            $K$-algebra of linear partial $q$-\textit{differential operators}
            with polynomial coefficients, respectively with rational
            coefficients is $K[t_1,\dotsc,t_n][D_1^{(q)},\dotsc,D_n^{(q)}]$,
            respectively the ring
            $K(t_1,\dotsc,t_n)[D_1^{(q)},\dotsc,D_n^{(q)}]$,
            subject to the relations:
            $t_jt_i=t_it_j$, $1\le i<j\le n$,
            $D_i^{(q)}t_i=qt_iD_i^{(q)}+1$, $1\le i\le n$,
            $D_j^{(q)}t_i=t_iD_j^{(q)}$, $i\neq j$,
            $D_j^{(q)}D_i^{(q)}=D_i^{(q)}D_j^{(q)}$, $1\le i<j\le n$. So $ q_{i}=q $, $ d_{i}=0 $ and $ a_{i}=1 $ for $ 1\leq i\leq n $.
        \end{enumerate}
    \item \label{multidime} The \textit{algebra $D$ for multidimensional discrete linear systems}.
    This algebra is a quasi-commutative bijective skew $ PBW $ extension of  $K[t_1,\dots,t_n]$, $D=\sigma(K[t_1,\dots,t_n])\langle x_1,\dots, x_n\rangle$.
    In fact, by definition, $D$ is the iterated skew polynomial ring $D:=K[t_1,\dots,t_n][x_1;\sigma_1]\cdots[x_n;\sigma_n]$, where $\sigma_i(p(t_1,\dots,t_n)):=
    p(t_1,\dots,t_{i-1},t_{i}+1,t_{i+1},\dots,t_n), \ \sigma_i(x_i)=x_i$, $1\leq i\leq n$. Thus, $ q_{i}=1 $, $ d_{i}=1 $ and $ a_{i}=0 $ for $ 1\leq i\leq n $.
    \item \label{additivean} \textit{Additive analogue of the Weyl algebra.} The $K$-algebra $A_n(q_1,\dots,q_n)$ is generated by
    $x_1,\dots,x_n,$ $y_1,\dots,y_n$ subject to the relations: $x_jx_i = x_ix_j$, $y_jy_i= y_iy_j$, $ 1 \leq i,j \leq n$,
    $y_ix_j=x_jy_i$, $i\neq j$, $y_ix_i= q_ix_iy_i + 1$, $1\leq i\leq n$, where $q_i\in K-\{0\}$. From these relations we have
        $A_n(q_1, \dots, q_n)\cong \sigma(K)\langle x_1,\dotsc,x_n;y_1,\dots,y_n\rangle\cong \sigma(K[x_1,\dotsc,x_n])\langle y_1,\dots,y_n\rangle$. So $ q_{i}\neq 0 $, $ d_{i}=0 $ and $ a_{i}=1 $ for $ 1\leq i\leq n $.
    \item The Weyl algebra $ A_{n}(K) $ (Example \ref{PBW} (b)), $ q_{i}=1 $, $ d_{i}=0 $ and $ a_{i}=1 $ for $ 1\leq i\leq n $.
    \item The algebra of $ q- $differential operators $ D_{q,h}[x,y] $ (Example \ref{ore}(a)), $ q_{1}=q $, $ d_{i}=0 $ and $ a_{1}=h $.
    \item The algebra of shift operators $ S_{h} $ (Example \ref{ore}(b)), $ q_{1}=1 $, $ d_{1}=-h $ and $ a_{i}=0 $ .

    \end{enumerate}
\end{examples}
\begin{examples}
    We have computed the center of the following algebras (our computations of course are agree with
    Theorem \ref{operator}, where the parameters $ q_{i}$ are roots of unity of degree $ l_{i} $):
    \begin{enumerate}
        \item The algebra of $ q- $differential operators $ D_{q,h}[x,y] $, $ q_{1}=q $, $ d_{i}=0 $ and $ a_{1}=h $, $ Z(A)=K[x^{l},y^{l}] $.
        \item Additive analogue of the Weyl algebra, $ q_{i}\neq 0 $, $ d_{i}=0 $ and $ a_{i}=1 $ for $ 1\leq i\leq n $, $ Z(A)=K[x_{1}^{l_{1}},\dots,x_{n}^{l_{n}},y_{1}^{l_{1}},\dots,y_{n}^{l_{n}}] $.
        \item Algebra of linear partial $ q- $dilation operators, $ q_{i}=q $, $ d_{i}=0 $ and $ a_{i}=0 $ for $ 1\leq i\leq n $, $Z(A)=K[t_{1}^{l},\dots,t_{n}^{l},H_{1}^{l},\dots,H_{n}^{l}] $, $ l=$degree$(q) $.
        \item Algebra of linear partial $ q- $differential operators, $ q_{i}=q $, $ d_{i}=0 $ and $ a_{i}=1 $ for $ 1\leq i\leq n $.
        $Z(A)=K[t_{1}^{l},\dots,t_{n}^{l},D_{1}^{l},\dots,D_{n}^{l}] $, $ l=$degree$(q) $.

    \end{enumerate}
\end{examples}

\section{Central subalgebras of other remarkable quantum algebras}
In this section we consider other noncommutative algebras of quantum type that can be interpreted
as skew $ PBW $ extensions, and for them, we compute either the center or some central key
subalgebras. These algebras are not covered by Thereom \ref{operator}. The proofs of the next
theorems are based in some lemmas which proofs we avoid since they can be realized by direct
computations.

\subsection{Woronowicz algebra $\cW_{\nu}(\mathfrak{sl}(2,K))$}
This algebra was
introduced by Woronowicz in \cite{Woronowicz} and it is generated
by $x,y,z$ subject to the relations
\begin{gather*}
xz-q^4zx=(1+q^2)x,\ \ \ xy-q^2yx=q z,\ \ \
zy-q^4yz=(1+q^2)y,
\end{gather*}
where $q \in K-\{0\}$ is not a root of unity. Then $\mathcal{W}_{q}(\mathfrak{sl}(2,K))\cong
\sigma (\textbf{\emph{K}}) \langle x,y,z\rangle$.
\begin{lemma} Let $ m\geq 1 $, then  in $\mathcal{W}_{q}(\mathfrak{sl}(2,K))$ the following identities hold:
    \begin{enumerate}
        \item $ y^{m}x= \frac{1}{q^{2m}}xy^{m}-\frac{\left( \sum_{i=0}^{m-1} q^{2i}\right)}{q}y^{m-1}z-
            \frac{\left(\sum_{i=1}^{m-1}iq^{2(i-1)}+\sum_{i=m}^{2m-2}((2m-1-i)q^{2(i-1)}  \right)}{q^{2m-1}}y^{m-1}
            $.
        \item $ yx^{m}= \frac{1}{q^{2m}}x^{m}y-\frac{\left(\sum_{i=0}^{m-1}q^{2i}  \right)}{q^{4m-3}}x^{m-1}z+
        \frac{\left( \sum_{i=1}^{m-1}iq^{2(i-1)}+\sum_{i=m}^{2m-2}(2m-1-i)q^{2(i-1)} \right)}{q^{4m-3}}x^{m-1}
        $.
        \item $ z^{m}x=\frac{\sum_{i=0}^{m}(-1)^{i}{m \choose i}(q^{2}+1)^{i}xz^{m-i}  }{q^{4m}}$.
        \item $ zx^{m}=\frac{1}{q^{4m}}x^{m}z+\frac{1}{q^{4m}}\left(\sum_{i=0}^{2m-1}q^{2i} \right)x^{m} $.
        \item $ z^{m}y= y(q^{4}z+(q^{2}+1))^{m}$.
        \item $ zy^{m}=q^{4m}y^{m}z+\sum_{i=0}^{2m-1}q^{2i}y^{m}$.
    \end{enumerate}
\end{lemma}

\begin{theorem} The Woronowicz algebra has trivial center,
    $Z(\mathcal{W}_{q}(\mathfrak{sl}(2,K)))=K$.
\end{theorem}
\begin{proof}
    Let $ f(x,y,z)=\sum_{i=0}^{n}C_{i}x^{\alpha_{i1}}y^{\alpha_{i2}}z^{\alpha_{i3}} \in Z(\mathcal{W}_{q}(\mathfrak{sl}(2,K)))$, with $ \alpha_{ij}\in \NN $, $ j=1,2,3 $. Since $ zf=fz $, then

    \small{\begin{align}
        \sum_{i=0}^{n}C_{i}x^{\alpha_{i1}}y^{\alpha_{i2}}z^{\alpha_{i3}+1}&=z\sum_{i=0}^{n}C_{i}x^{\alpha_{i1}}y^{\alpha_{i2}}z^{\alpha_{i3}} \\
        &=\sum_{i= 0}^n C_i\left(\frac{1}{q^{4\alpha_{i1}}}x^{\alpha_{i1}}zy^{\alpha_{i2}}z^{\alpha_{i3}} \right)+\frac{C_i}{q^{4\alpha_{i1}}}\left(\sum_{j=0}^{2\alpha_{i1}-1}q^{2j} \right)x^{\alpha_{i1}}y^{\alpha_{i2}}z^{\alpha_{i3}} \label{2.2}\\
        &=\sum_{i= 0}^n \frac{C_i}{q^{4\alpha_{i1}}}\left[\left(x^{\alpha_{i1}}\left(q^{4\alpha_{i2}}y^{\alpha_{i2}}z+\sum_{j=0}^{2\alpha_{i2}-1} q^{2j}y^{\alpha_{i2}}\right) z^{\alpha_{i3}} \right)+\left( \sum_{j=0}^{2\alpha_{i1}-1}q^{2j} \right)x^{\alpha_{i1}}y^{\alpha_{i2}}z^{\alpha_{i3}}\right]\label{2.3}\\
        &=\sum_{i=0}^n C_i\frac{q^{4\alpha_{i2}}}{q^{4\alpha_{i1}}}x^{\alpha_{i1}}y^{\alpha_{i2}} z^{\alpha_{i3}+1} +\frac{C_i}{q^{4\alpha_{i1}}}
        \left( \sum_{j=0}^{2\alpha_{i2}-1}q^{2j}+\sum_{j=0}^{2\alpha_{i1}-1}q^{2j}
        \right)x^{\alpha_{i1}}y^{\alpha_{i2}}z^{\alpha_{i3}}.
        \end{align}}
    \noindent
    Equations ($\ref{2.2}$) and ($\ref{2.3}$) are the relations $ 4 $ and $ 6 $ in the previous lemma, respectively.  From this,
    $ C_i=C_i\frac{q^{4\alpha_{i2}}}{q^{4\alpha_{i1}}} $ and $\frac{C_i}{q^{4\alpha_{i1}}}\left( \sum_{j=0}^{2\alpha_{i2}-1}q^{2j}+ \sum_{j=0}^{2\alpha_{i1}-1}q^{2j}
    \right)=0$, so $ 1=\frac{q^{4\alpha_{i2}}}{q^{4\alpha_{i1}}}=0 $, whence $ \alpha_{i1}=\alpha_{i2} $, and since $ char(K)=0 $, $ C_i=0 $,
     or, $ \sum_{j=0}^{2\alpha_{i1}-1}q^{2j}=0 $, but $ q $ is not a root of unity, then $\sum_{i=0}^{2\alpha_{i1}-1}q^{2j}\neq 0$, thus $ C_{i}=0 $ for
     $ \alpha_{i1},\alpha_{i2}>0 $, i.e., $ f=\sum_{i=0}^nC_{i}z^{\alpha_{i3}} $. But $ yf=fy $, so $ \sum_{i=0}^nC_iyz^{\alpha_{i3}}=
     \sum_{i=0}^nC_iz^{\alpha_{i3}}y=\sum_{i=0}^nC_iy( q^{4}z+(q^{2}+1))^{\alpha_{i3}} $, then $ C_i=C_iq^{4\alpha_{i3}} $ and
     $ \sum_{j=1}^{\alpha_{i3}} {\alpha_{i3} \choose j}q^{4(\alpha_{i3}-j)}(q^{2}+1)^{j}=0 $, whence $ q^{4\alpha_{i3}}=1 $, thus $ \alpha_{i3}=0 $ and
     therefore $ f=C_{0}\in K $. Finally, $ Z(\mathcal{W}_{q}(\mathfrak{sl}(2,K)))\subseteq K\subseteq Z(\mathcal{W}_{q}(\mathfrak{sl}(2,K))) $.
\end{proof}


\subsection{The algebra U}
Let $U$ be the algebra generated over the field $K=\mathbb{C}$ by the set of variables
$x_i,y_i,z_i$, $1\leq i\leq n$, subject to relations:
\begin{align*}
x_jx_i &=x_ix_j,\ y_jy_i =y_iy_j,\ z_jz_i=z_iz_j, \ 1\le i,j\le n,\\
y_jx_i&=q^{\delta_{ij}}x_iy_j,\ z_jx_i =q^{-\delta_{ij}}x_iz_j, \ 1\le i,j \le n,\\
z_jy_i &=y_iz_j, \ i\neq j,\ z_iy_i =q^{2}y_iz_i-q^2x_i^2, \ 1\le
i\le n,
\end{align*}
where $q\in \mathbb{C}-\{0\}$. Note that $U$ is a bijective skew $PBW$ extension of
$\mathbb{C}[x_1,\dotsc,x_n]$, that is,     $ U\cong\sigma(\mathbb{C}[x_1,\dotsc,x_n])\langle
y_1,\dotsc,y_n;z_1,\dotsc,z_n\rangle$.

\begin{lemma}
For $ m\geq 1 $ and $ i=1,\dots,n $, in  $ U $ the following relations hold:
\begin{enumerate}
\item   $z_iy_i^m=q^{2m}y_i^mz_i-mq^2x_i^2y_i^{m-1}$,
\item   $z_i^my_i=q^{2m}y_iz_i^m-mq^{2m}x_i^2z_i^{m-1}$,
\item   $y_i^mx_i=q^mx_iy_i^m$,
\item   $z_i^mx_i=q^{-m}x_iz_i^m$.
\end{enumerate}

\end{lemma}
\begin{theorem}
    If $q$  is not a root of unity, then the algebra $ U $ has a trivial center $ K $.
    If $q$ is a root of unity of degree $l\geq 2$, $x_1^l,\dots,x_n^l$ are central elements in $ U$ and  $ K[x_1^l,\dots,x_n^l]\subseteq Z(U) $.
\end{theorem}
\begin{proof}

        Let $ f=\sum_{i=1}^{t}r_ix_1^{\gamma_{i1}}\cdots x_n^{\gamma_{in}}y_1^{\alpha_{i1}}\cdots y_n^{\alpha_{in}}z_1^{\beta_{i1}}\cdots z_n^{\beta_{in}}
    \in Z(U)$, $ \gamma_{ij},\alpha_{ij},\beta_{ij}\geq 0 $, $ j=1,\dots,n $. Then $ x_1f=fx_1 $, and hence
    \begin{center}
    $ \sum_{i=1}^{t}r_ix_1^{\gamma_{i1}+1}\cdots x_n^{\gamma_{in}}y_1^{\alpha_{i1}}\cdots y_n^{\alpha_{in}}z_1^{\beta_{i1}}\cdots z_n^{\beta_{in}}=
    \sum_{i=1}^{t}r_iq^{\alpha_{i1}-\beta_{i1}}x_1^{\gamma_{i1}+1}\cdots x_n^{\gamma_{in}}y_1^{\alpha_{i1}}\cdots y_n^{\alpha_{in}}z_1^{\beta_{i1}}
    \cdots z_n^{\beta_{in}} $.
    \end{center}
    Thus, $ 1-q^{\alpha_{i1}-\beta_{i1}}=0 $, so $ \alpha_{i1}=\beta_{i1} $. Repeating this for $ x_2,\dots,x_n $ we conclude
    that $ \alpha_{ij}=\beta_{ij} $, $ j=1,\dots,n $. Therefore,
    $ f=\sum_{i=1}^{t}r_ix_1^{\gamma_{i1}}\cdots x_n^{\gamma_{in}}y_1^{\alpha_{i1}}\cdots y_n^{\alpha_{in}}\cdots z_n^{\alpha_{in}} $.
    But $ fz_1=z_1f $, then
    \begin{center}
    $ \sum_{i=1}^{t}r_ix_1^{\gamma_{i1}}\cdots x_n^{\gamma_{in}}y_1^{\alpha_{i1}}\cdots y_n^{\alpha_{in}}z_1^{\alpha_{i1}+1}\cdots z_n^{\alpha_{in}}=$
    $\sum_{i=1}^{t}r_iq^{2\alpha_{i1}-\gamma_{i1}}x_1^{\gamma_{i1}}\cdots x_n^{\gamma_{in}}y_1^{\alpha_{i1}}\cdots$ $ y_n^{\alpha_{in}}z_1^{\alpha_{i1}+1}
    \cdots z_n^{\alpha_{in}}+\sum_{i=1}^{t}\alpha_{i1}r_iq^{2\alpha_{i1}-\gamma_{i1}+2}x_1^{\gamma_{i1}+2}\cdots x_n^{\gamma_{in}}y_1^{\alpha_{i1}-1}
    \cdots y_n^{\alpha_{in}}z_1^{\alpha_{i1}}\cdots z_n^{\alpha_{in}} $,
    \end{center}
    so necessarily $ 1-q^{2\alpha_{i1}-\gamma_{i1}}=0 $ and
    $ \alpha_{i1}r_iq^{2\alpha_{i1}-\gamma_{i1}+2}=0 $, whence $ 2\alpha_{i1}=\gamma_{i1} $ and $ \alpha_{i1}=0 =\gamma_{i1}$, so
    $ f=\sum_{i=1}^{t}r_ix_2^{\gamma_{i2}}\cdots x_n^{\gamma_{in}}y_2^{\alpha_{i2}}\cdots y_n^{\alpha_{in}}z_2^{\alpha_{i2}}\cdots z_n^{\alpha_{in}}
    $; repeating this for $ z_2,\dots,z_n $ then we obtain that $ f=r\in K $. Therefore, $ K\subseteq Z(U)\subseteq K $.

    If $ q^l=1 $, then $ y_jx_i^l=x_i^ly_j $, $ z_jx_i^l=x_i^lz_j $, and $ y_ix_i^l=q^lx_i^ly_i=x_i^ly_i $, $ z_ix_i^l=q^{-l}x_i^lz_i=x_i^lz_i $,
    for $ 1\leq i\leq n $. Thus, $ K[x_1^l,\dots,x_n^l]\subseteq Z(U) $.
\end{proof}


\subsection{Dispin Algebra $ \mathcal{U}(osp(1,2)) $}
It is generated by \noindent $ x,y,z $ over $ K $ satisfying the relations:
\begin{align*}
yx&=xy-x,&
zx&=-xz+y,&
zy&=yz-z.\\
\end{align*}
Then, $ \mathcal{U}(osp(1,2)) $ is a skew $ PBW $ extension of $ K $, $ \mathcal{U}(osp(1,2))\cong
\sigma(K)\langle x,y,z\rangle $. Using the defining relation of the algebra $ \mathcal{U}(osp(1,2))
$ we get the following preliminary result.
\begin{lemma}
    For $ n,m \geq 1$,
    \begin{enumerate}
        \item $ yx^{n}=x^{n}y-nx^{n} $.
        \item $ y^{n}x=x(y-1)^{n} $.
        \item $ y^{n}x^{m}=x^{m}(y-m)^{n} $.
        \item $ z^{n}x=(-1)^{n}xz^{n}-\frac{n}{2}z^{n-1}+\left( \frac{1+(-1)^{n+1}}{2}\right)yz^{n-1} $
        \begin{enumerate}
            \item[4.1] If $ n=2l  $, $ z^{2l}x=xz^{2l}-lz^{2l-1} $.
            \item[4.2] If $ n=2l+1 $, $ z^{2l+1}x=-xz^{2l+1}+yz^{2l}-nz^{2l} $.
        \end{enumerate}
        \item $ zy^{n}=(y-1)^{n}z $.
        \item $ z^{n}y=(y-n)z^{n} $.
        \item $ z^{n}y^{m}=(y-n)^{m}z^{n} $.
    \end{enumerate}
\end{lemma}
\begin{theorem}
    The element  $ f=4x^{2}z^{2}-y^{2}-2xz-y $ is central in $ \mathcal{U}(osp(1,2)) $ and $ K[f]\subseteq Z(\mathcal{U}(osp(1,2))) $.
\end{theorem}
\begin{proof}
    Since $ fx=4x^{2}z^{2}x-y^{2}x-2xzx-yx=4x^{2}(xz^{2}-z)-x(y-1)^2-2x(-xz+y)-(xy-x)=4x^3z^2-4x^2z-xy^2+2xy-x+2x^2z-2xy-xy+x=4x^{3}z^{2}-xy^{2}-2x^2z-xy=xf $.
    $ yf=4yx^{2}z^{2}-y^{3}-2yxz-y^2=4(x^{2}y-2x^2)z^2-y^3-2(xy-x)z-y^2=4x^2yz^2-8x^2z^2-y^3-2xyz-2xz-y^2 $. Whereas
    $ fy=4x^{2}z^{2}y-y^{3}-2xzy-y^2=4x^2(yz^2-2z^2)-y^3-2x(yz-z)-y^2=4x^2yz^2-8x^2z^2-y^3-2xyz-2xz-y^2 $. Thus, $ yf=fy $.
    $zf=4zx^{2}z^{2}-zy^{2}-2zxz-zy=4(x^2z-x)z^2-(y^2z-2yz+z)-2(-xz+y)z-(yz-z)=4x^2z^3-4xz^2-y^2z+2yz-z+2xz^2-2yz-yz+z=4x^{2}z^{3}-y^{2}z-2xz^2-yz=fz   $. Therefore, $ f\in Z(\mathcal{U}(osp(1,2))) $.
\end{proof}


\subsection{$ q- $Heisenberg algebra $ H_{n}(q) $}
This algebra is generated by  $ 3n $ variables, $ x_{1},\dots,x_{n} $, $ y_{1},\dots,y_{n} $, $
z_{1},\dots,z_{n} $ with relations:
\begin{align*}
x_{j}x_{i}=&x_{i}x_{j},&y_{j}y_{i}=&y_{i}y_{j},&z_{j}z_{i}=&z_{i}z_{j}, &i\neq &j.\\
y_{j}x_{i}=&x_{i}y_{j},&z_{i}y_{j}=&y_{j}z_{i},&z_{j}x_{i}=&x_{i}z_{j}& 1\leq &i<j\leq n,\\
y_{i}x_{i}=&qx_{i}y_{i},&z_{i}y_{i}=&qy_{i}z_{i},&z_{i}x_{i}=&q^{-1}x_{i}z_{i}+y_{i},
\end{align*}
with $q\in K-\{0\}$. Note that
\begin{align*}
H_n(q) \cong \sigma(K)\langle
x_1,\dots,x_n;y_1,\dots,y_n;z_1,\dots,z_n\rangle \cong \sigma(K[y_1,\dotsc,y_n])\langle
x_1,\dotsc,x_n;z_1,\dotsc,z_n\rangle.
\end{align*}

\begin{theorem} Let $ C_{i}=(q^{2}-1)x_{i}y_{i}z_{i}-y_{i}^{2} $, for $ i=i,\dots,n $. Then
    \begin{enumerate}
        \item If $ q $ is not a root of unity, then the elements $ C_i $ are central,   ${1\leq i\leq n} $,
        and
        $K[ C_{i} ]_{1\leq i\leq n}\subseteq Z(H_{n}(q))$.
        \item If $ q $ is a root of unity of degree $ l $, the following elements are central, $ C_{i},x_{i}^{l},y_{i}^{l},z_{i}^{l} $, $ i=1,\dots,n $, and
        $K[ C_{i},x_{i}^{l},y_{i}^{l},z_{i}^{l} ]_{1\leq i\leq n}\subseteq Z(H_{n}(q))$.
    \end{enumerate}
\end{theorem}
\begin{proof}
    It is clear that $ x_{j}C_{i}=C_{i}x_{j} $, $ y_{j}C_{i}=C_{i}y_{j} $, $ z_{j}C_{i}=C_{i}z_{j} $, for $  i\neq j $.
    $ C_{i}x_{i}= (q^{2}-1)x_{i}y_{i}z_{i}x_{i}-y_{i}^{2}x_{i}=(q^{2}-1)x_{i}y_{i}(q^{-1}x_{1}z_{i}+y_{i})-y_{i}qx_iy_i=q^{-1}
    (q^2-1)x_i(qx_iy_i)z_i+(q^2-1)x_iy_i^2-q^2x_iy_i^2=(q^2-1)x_i^2y_iz_i-x_iy_i^2=x_iC_i$. $ z_iC_{i}=(q^{2}-1)z_ix_{i}y_{i}z_{i}-z_iy_{i}^{2}=
    (q^2-1)(q^{-1}x_iz_i+y_i)y_iz_i-q^2y_i^2z_i =(q^2-1)q^{-1}x_i(qy_iz_i)z_i+(q^2-1)y_i^2z_i-q^2y_iz_i=(q^2-1)x_iy_iz_i^2-y_i^2z_i=C_iz_i$, and
    $ y_iC_i=(q^2-1)y_ix_iy_iz_i-y_i^3=(q^2-1)qx_iy_i^2z_i-y_i^3 $, whereas $ C_iy_i=(q^2-1)x_iy_iz_iy_i-y_i^3=(q^2-1)qx_iy_i^2z_i-y_i^3 $, so $ y_iC_i=C_iy_i
   $, and therefore, $ C_i $, $ 1\leq i\leq n $ are central elements in $ H_n(q) $.
\end{proof}


\subsection{The coordinate algebra of the quantum matrix space, $ M_{q}(2) $.}
This is the $ K- $algebra generated by the variables $ x,y,u,v $ satisfying the relations:
\begin{equation}\label{3com}
xu=qux,\ \ \ \ \ \ yu=q^{-1}uy,\ \ \ \ \ \ vu=uv,
\end{equation}
and
\begin{equation}\label{4com}
xv=qvx,\ \ \ \ \ \ \ vy=qyv,\ \ \ \ \ \ \ \ yx-xy=-(q-q^{-1})uv,
\end{equation}
with $ q\in K-\{0\} $.  Thus, $\mathcal{O}(M_q(2))\cong \sigma(K[u])\langle x,y,v\rangle$. Due to
the last relation in (\ref{4com}), we remark that it is not possible to consider
$\mathcal{O}(M_q(2))$ as a skew $PBW$ extension of $K$. This algebra can be generalized to $n$
variables, $\mathcal{O}_q(M_n(K))$, and coincides with the \textit{coordinate algebra of the
    quantum group $SL_q(2)$}, see \cite{Gomez-Torrecillas2} for more details.
\begin{lemma}
In $\mathcal{O}(M_q(2))$ the following relations hold:
\begin{enumerate}
    \item $ yx^{l}=x^{l}y-(q-q^{-1})(q^{2(l-1)}+q^{2(l-2)}+\dots+q^{2}+1)uvx^{l-1} $.
    In particular, if  $ q $ is a root of unity of degree $ l\geq 3 $,  then $ q^{2(l-1)}+q^{2(l-2)}+\dots+q^{2}+1=0 $ and $ yx^{l}=x^{l}y $.
    \item $ y^{l}x=xy^{l}-(q-q^{-1})(1+q^{-2}+\dots+q^{-2(l-2)}+q^{-2(l-1)})uvy^{l-1} $. Where $ 1+q^{-2}+\dots+q^{-2(l-2)}+q^{-2(l-1)}=0 $ if $ q $ is a root of unity of degree $ l $.
\end{enumerate}

\end{lemma}

With the previous relations it is possible proof the following theorem.
\begin{theorem}

    Let $ A=M_{q}(2) $, then
    \begin{enumerate}
        \item   $K[ xy-quv ]\subseteq Z(A)$ if $q$  is not a root of unity.
        \item $ x^{l},y^{l},xy-quv,u^{i}v^{j} $, with $ i+j=l $, are central elements if $ q $ is a root of unity of degree $ l
        $, and
        $K[x^{l},y^{l},xy-quv,u^{i}v^{j}]\subseteq Z(A).$
    \end{enumerate}
\end{theorem}
\begin{proof}
    Let $ f=xy-quv $, then
    $ fu=(xy-quv)u=xyu-quvu=xq^{-1}uy-qu^2v=q^-1quxy-qu^2v=uxy-qu^{2}v=uf $. $ vf=v(xy-quv)=vxy-qvuv=vxy-quv^2 $, and
    $ fv=(xy-quv)v=xyv-quv^2=xq^{-1}vy-quv^2=q^{-1}qvxy-quv^2=vxy-quv^2 $, so $ vf=fv $. Also, $ xf=x^2y-qxuv=x^2y-q(qux)v=x^2y-q^3uvx $,
    $ fx=xyx-quvx=x(xy-(q-q^{-1})uv)-quvx=x^2y-(q-q^{-1})xuv-quvx=x^2y-(q-q^{-1})q^2uvx=x^2y-q^3uvx+quvx-quvx=x^2y-q^3uvx=xf $; $ fy=xy^2-quvy $, $ yf=yxy-qyuv=xy^2-(q-q^{-1})uvy-q^{-1}uvy=xy^2-quvy+q^{-1}uvy-q^{-1}uvy=xy^2-quvy=fy $.\\
    For the second part it is clear from the previous relations that $ yx^l=x^ly $ and $ y^lx=xy^l $ if $ q $ is a root of unity of degree $ l $.  Also, $ y^lu=q^{-l}uy^l=uy^l $, $ y^lv=q^{-l}vy^l=vy^l $ and $ x^lu=q^{l}ux^l=ux^l $, $ x^lv=q^{l}vx^l=vx^l $, thus $ x^l $, $y^l$ are central elements in $ A $. For $ u^{i}v^{j} $ with $ i+j=l $ we have, $ xu^iv^j=q^{i}u^ixv^j=q^iq^ju^iv^jx=u^iv^jx $, $ yu^iv^j=q^{-i}u^iyv^j=q^{-i}q^{-j}u^iv^jy=u^iv^jy $, so $ u^iv^{j} $ are central for $ i+j=l $.
\end{proof}


\subsection{Quadratic algebras in 3 variables}
For quadratic algebras in 3 variables the relations are homogeneous of degree 2. More exactly, a
quadratic algebra in 3 variables $\mathcal{A}$ is a $K$-algebra generated by $x,y,z$ subject to the
relations
\begin{align*}
yx & = xy+a_1xz+a_2y^2+a_3yz+\xi_1z^2,\\
zx & = xz+\xi_2y^2+a_5yz+a_6z^2,\\
zy & =yz+a_4z^2.
\end{align*}
If $ a_1=a_4=0 $ we obtain the relations
\begin{align*}
yx & = xy+a_2y^2+a_3yz+\xi_1z^2,\\
zx & = xz+\xi_2y^2+a_5yz+a_6z^2,\\
zy & =yz.
\end{align*}
One can check that $ \mathcal{A}_1\cong \sigma(K[y,z])\langle x\rangle $. If $ a_3=a_5=0 $, which
implies $ a_2=a_6=0 $, and thus, there is a family of algebras with relations
\begin{align*}
yx & = xy+a_1xz+\xi_1z^2,\\
zx & = xz,\\
zy & =yz+a_4z^2.
\end{align*}
These algebras are skew $ PBW $ extensions of the form $ \mathcal{A}_2\cong\sigma( K[x,z])\langle y\rangle $.
If $ a_1=a_3=\xi_1=0 $, then $ a_4=a_5=a_6=0 $ and thus there is a family of algebras with relations
\begin{align*}
yx & = xy+a_2y^2,\\
zx & = xz+\xi_2y^2,\\
zy & =yz.
\end{align*}
These algebras are skew $ PBW $ extensions of the form $ \mathcal{A}_3\cong\sigma( K[x,y])\langle z\rangle $.

\begin{lemma} Assume that in the previous algebra $ \mathcal{A}_2 $, $ a_1=a $, $ \xi_1=b $ and $ a_4=c $. Then for $ m\geq 1 $
    \begin{enumerate}
        \item $ yx^m=x^my+max^mz+mbx^{m-1}z^2 $.
        \item $ zy^m=\sum_{i=0}^{m}\dfrac{m!}{(m-i)!}c^{i}y^{m-i}z^{1+i} $.
        \item $ z^my=yz^m+mcz^{m+1} $.
    \end{enumerate}
\end{lemma}

\begin{theorem}
    The quadratic algebra $ \mathcal{A}_2 $ generated by $ x,y,z $ with relations $ yx = xy+axz+bz^2$, $zx  = xz$, $zy  =yz+cz^2$, $ a,b,c\in K-\{0\} $ and $ ac<0 $ has trivial center.
\end{theorem}
\begin{proof}
        Let $ f(x,y,z)=\sum_{i=0}^{n}C_{i}x^{\alpha_{i1}}y^{\alpha_{i2}}z^{\alpha_{i3}} \in Z(\mathcal{A}_2)$, with $ \alpha_{ij}\in \NN $, $ j=1,2,3 $.
        Since $ zf=fz $, then $ \sum_{i=0}^{n}C_{i}x^{\alpha_{i1}}y^{\alpha_{i2}}z^{\alpha_{i3}+1}=\sum_{i=0}^{n}C_{i}x^{\alpha_{i1}}y^{\alpha_{i2}}z^{\alpha_{i3}+1}+
        \sum_{i=0}^{n}C_{i} \sum_{j=1}^{\alpha_{i2}}\frac{(\alpha_{i2})!}{(\alpha_{i2}-j)!}c^j   x^{\alpha_{i1}}y^{\alpha_{i2}-j}z^{\alpha_{i3}+1-j} $,
        so $ \sum_{i=0}^{n}C_{i} \sum_{j=1}^{\alpha_{i2}}\frac{(\alpha_{i2})!}{(\alpha_{i2}-j)!}c^j   x^{\alpha_{i1}}y^{\alpha_{i2}-j}z^{\alpha_{i3}+1-j}=0 $,
        but char$ (K) =0$, then $ C_i=0 $ for $ \alpha_{i2}\neq 0 $, i.e. $ f=\sum_{i=0}^{n}C_ix^{\alpha_{i1}}z^{\alpha_{i3}} $. Also,
        $ yf=fy $, this is $ \sum_{i=0}^{n}C_ix^{\alpha_{i1}}yz^{\alpha_{i3}}+C_i\alpha_{i1}ax^{\alpha_{i1}}z^{\alpha_{i3}+1}+
        C_i\alpha_{i1}bx^{\alpha_{i1}-1}z^{\alpha_{i3}+2}=\sum_{i=0}^{n}C_ix^\alpha_{i1}yz^{\alpha_{i3}}+C_i\alpha_{i3}cx^{\alpha_{i1}}z^{\alpha_{i3}+1} $. Thus,
        $ C_i(\alpha_{i1}a-\alpha_{i3}c)x^{\alpha_{i1}}z^{\alpha_{i3}+1}+C_i\alpha_{i1}bx^{\alpha_{i1}-1}z^{\alpha_{i3}+2}=0 $ for $ i=0,1,\dots,n $. Since
        $ \alpha_{i1}a-\alpha_{i3}c\leq0 $ or $ \alpha_{i1}a-\alpha_{i3}c\geq0 $,  then $ \alpha_{i1}=0 $, $ \alpha_{i3}=0 $ and $ f=c_0\in K $, therefore $ K\subseteq Z(\mathcal{A}_2)\subseteq K $.
\end{proof}
\noindent For $ \cA_3 $ it is not difficult to prove the following result.

\begin{theorem}
    The element $ \xi_2y-a_2z $ is central in the quadratic algebra $ \cA_3 $.
\end{theorem}


\subsection{Witten's deformation of $\mathcal{U}(\mathfrak{sl}(2,K)$}
Witten introduced and studied a 7-parameter deformation of the universal enveloping
algebra $\mathcal{U}(\mathfrak{sl}(2,K))$ depending on a 7-tuple of parameters $\underline{\xi}=(\xi_1,\dotsc,\xi_7)$ and subject
to relations
\[
xz-\xi_1zx=\xi_2x,\ \ \ zy-\xi_3yz=\xi_4y,\ \ \ \ yx-\xi_5xy=\xi_6z^2+\xi_7z.
\]
The resulting algebra is denoted by $W(\underline{\xi})$. In \cite{Levandovskyy} it is assumed that
$\xi_1\xi_3\xi_5\neq 0$. Note that $W(\underline{\xi})\cong \sigma(\sigma(K[x])\langle
z\rangle)\langle y\rangle$.
The computation of the center of this algebra is in general a difficult task; however, in the
particular case of the \textit{conformal} \textit{algebra} considered in \cite{Levandovskyy}, the
center is trivial. This particular algebra is generated by $ x,y,z $ subject to relations
\begin{align*}
yx=&cxy+bz^2+z\\
zx=&\frac{1}{a}xz-\frac{1}{a}x\\
zy=&ayz+y
\end{align*}
with $ a,b,c\in K-\{0\} $.
    \begin{theorem}
        The conformal algebra has trivial center.
    \end{theorem}
    \begin{proof}
        Note that $ Z(Gr( \mathfrak{sl}(2,K))) =K$, but $ Gr(Z( \mathfrak{sl}(2,K)))\subseteq Z(Gr( \mathfrak{sl}(2,K))) $, so $ Z( \mathfrak{sl}(2,K))=K $.
    \end{proof}

\subsection{Algebra $ \cD $ of diffusion type}

In \cite{diffusionauthor} were introduced the \textit{diffusion algebras}; following this notion we
define for $n\geq 2$ the algebra $ \cD $ which is generated by $ \{D_i,x_i\mid 1\leq i\leq n \} $
over $ K $ with relations:
 \[ x_ix_j=x_jx_i , \quad x_iD_j=D_jx_i , \quad  1\leq i,j\leq n ;\]
 \[ c_{i,j}D_iD_j-c_{j,i}D_jD_i=x_jD_i-x_iD_j,\quad  i<j ,\]  $c_{i,j},c_{j,i}\in K^* $. Observe
 that
 $ A\cong \sigma(K[x_1,\dots,x_n])\langle D_1,\dots,D_n \rangle $ is a bijective skew $ PBW $ extension of $ K[x_1,\dots,x_n]
 $, and of course we obtain the following trivial result.

    \begin{theorem}
        $ K[x_1,\dots,x_n] \subseteq Z(\cD)$.
    \end{theorem}

\section{Summary and applications}

\subsection{Tables}
In this subsection we summarize in tables the description of the center of some remarkable examples
of skew $ PBW $ extensions
studied in the previous sections. Recall that $K$ is a field with $char(K)=0$.\\
\begin{table}[H]\label{table4.1}
    \centering \tiny{
        \begin{tabular}{|l|c|}\hline 
            \textbf{Algebra} & $Z(A)$
            \\ \hline
            \cline{1-2} Weyl algebra $A_n(K)$ & $K$   \\ \cline{1-2}
            Extended Weyl algebra $B_n(K)$ & $K$  \\
            \cline{1-2} Universal enveloping algebra of $ \mathfrak{sl}(2,K) $ & $ K[4xy+z^{2}-2z] $\\
            \cline{1-2} Universal enveloping algebra of $S\mathfrak{o}(3,K)$ & $ K[x^2+y^2+z^2] $\\
            \cline{1-2} Tensor product $R\otimes_K \cU(\cG)$ &  $ Z(R)\otimes_K Z(\cU(\cG)) $  \\
            \cline{1-2} 
            \cline{1-2} Algebra of q-differential operators $D_{q,h}[x,y]$
            & $K$  \\
            \cline{1-2} Algebra of shift operators $S_h$ & $K$   \\
            \cline{1-2}
            Mixed algebra $D_h$ & $K$   \\
            \cline{1-2} Discrete linear systems
            $K[t_1,\dotsc,t_n][x_1,\sigma_1]\dotsb[x_n;\sigma_n]$ & $K$
            \\
            \cline{1-2} Linear partial shift operators
            $K[t_1,\dotsc,t_n][E_1,\dotsc,E_n]$ & $K$   \\
            \cline{1-2} 
            \cline{1-2} L.P. Differential operators
            $K[t_1,\dotsc,t_n][\partial_1,\dotsc,\partial_n]$ & $K$   \\
            \cline{1-2} 
            \cline{1-2} L. P. Difference operators
            $K[t_1,\dotsc,t_n][\Delta_1,\dotsc,\Delta_n]$ & $K$   \\
            \cline{1-2} 
            \cline{1-2} L. P. $q$-dilation operators
            $K[t_1,\dotsc,t_n][H_1^{(q)},\dotsc,H_m^{(q)}]$ & $K$   \\
            \cline{1-2} 
            \cline{1-2} L. P. $q$-differential operators
            $K[t_1,\dotsc,t_n][D_1^{(q)},\dotsc,D_m^{(q)}]$ & $K$   \\
            \cline{1-2} 
            \cline{1-2} 
            Additive analogue of the Weyl algebra $A_n(q_1,\dotsc,q_n)$ & $K$
            \\
            \cline{1-2} Multiplicative analogue of the Weyl algebra $\cO_n(\lambda_{ji})$ & $K$   \\ \cline{1-2}
            $ n- $multiparametric quantum space & $ K $\\ \hline
            Quantum algebra $\cU'(\mathfrak{so}(3,K))$ & $K[-q^{1/2}(q^{2}-1)I_{1}I_{2}I_{3}+q^{2}I_{1}^{2}+I_{2}^{2}+q^{2}I^{2}_{3}]$  \\
            \cline{1-2} Woronowicz algebra $\cW_{\nu}(\mathfrak{sl}(2,K))$ &
            $K$   \\
            \cline{1-2} Algebra \textbf{U} &  $K$  \\
            \cline{1-2} Quantum enveloping algebra of
            $\mathfrak{sl}(2,K)$, $\cU_q(\mathfrak{sl}(2,K))$ &  $ K[(q^2-1)^2EF+qK_e+q^3K_f] $  \\ \cline{1-2}
            Differential operators on a quantum space $S_{\textbf{q}}$,
            $D_{\textbf{q}}(S_{\textbf{q}})$   & $ K $  \\ \cline{1-2}
            Particular Witten's Deformation of $\cU(\mathfrak{sl}(2,K)$ & $K $
            \\
            \cline{1-2} Quantum Weyl algebra of Maltsiniotis
            $A_n^{\textbf{q},\lambda}$ &  $K$ \\
            \cline{1-2}
            Multiparameter Weyl algebra $A_n^{Q,\Gamma}(K)$ & $K$\\ \cline{1-2}
            Quantum symplectic space $\cO_q(\mathfrak{sp}(K^{2n}))$ & $K$ \\
            \cline{1-2}
            Jordan plane $ \mathcal{J} $ & $ K $\\ \cline{1-2}
            Quantum plane & $ K $\\ \cline{1-2}
            Quadratic algebras in 3 variables, $ \mathcal{A}_2 $ & $K$   \\
            \cline{1-2}
    \end{tabular}}
    \caption{Center of some bijective skew $PBW$
        extensions which parameters $ q $'s are not roots of unity.}\label{table4.1}
\end{table}

\begin{table}[H]\label{table4.2}
    \centering \tiny{
        \begin{tabular}{|l|c|}\hline 
            \textbf{Algebra} & $Z(A)$
            \\ \hline
            \hline
            Quantum plane & $ K[x^n,y^n] $\\
            \hline 
            Algebra of q-differential operators $D_{q,h}[x,y]$  & $K[x^{n},y^{n}]$  \\
            \hline
            L. P. $q$-dilation operators $K[t_1,\dotsc,t_n][H_1^{(q)},\dotsc,H_m^{(q)}]$ & $K[t_1^{l},\dotsc,t_n^{l},H^{l}_1,\dotsc,H^{l}_m]$   \\
            \hline
            L. P. $q$-differential operators    $K[t_1,\dotsc,t_n][D_1^{(q)},\dotsc,D_m^{(q)}]$ & $K[t^{l}_1,\dotsc,t^{l}_n,D_1^{l},\dotsc,D_m^{l}]$   \\
            \hline
            Additive analogue of the Weyl algebra $A_n(q_1,\dotsc,q_n)$ & $K[x_{1}^{\alpha_{1}},y_{1}^{\alpha_{1}},\dotsc,x_{n}^{\alpha_{n}},y_{n}^{\alpha_{n}}]$
            \\
            \cline{1-2} Multiplicative analogue of the Weyl algebra $\cO_n(\lambda_{ji})$ & $K[x_{1}^{L_{1}},\dots,x_{n}^{L_{n}}]$   \\ \cline{1-2}
            $ n- $multiparametric quantum space & $K[x_{1}^{L_{1}},\dots,x_{n}^{L_{n}}]$   \\ \cline{1-2}
            Quantum algebra $\cU'(\mathfrak{so}(3,K))$ & $K[-q^{1/2}(q^{2}-1)I_{1}I_{2}I_{3}+q^{2}I_{1}^{2}+I_{2}^{2}+q^{2}I^{2}_{3},C_{n}^{(1)},C_{n}^{(2)},C_{n}^{(3)}]$  \\
            \cline{1-2} Quantum enveloping algebra of
            $\mathfrak{sl}(2,K)$, $\cU_q(\mathfrak{sl}(2,K))$ &  $ K[(q^2-1)^2EF+qK_e+q^3K_f,E^n,F^n,K_e^n,K_f^n] $  \\ \hline
            Quantum symplectic space $\cO_q(\mathfrak{sp}(K^{2n}))$ &  $ K[x_1^m,\dots,x_{2n}^m] $\\
            \hline
    \end{tabular}}
    \caption{Center of some bijective skew $PBW$
        extensions which parameters $ q $'s are  roots of unity.}\label{table4.2}
\end{table}

\begin{table}[H]\label{table4.3}
    \centering \tiny{
        \begin{tabular}{|l|c|}\hline 
            \textbf{Algebra} & $\subseteq Z(A)$
            \\ \hline
            \cline{1-2} Algebra $\cD$ of diffusion type &     $ K[x_1,\cdots,x_n] $ \\ \cline{1-2}Dispin algebra $\cU(osp(1,2))$ & $\mathbb{K}[4x^{2}z^{2}-y^{2}-2xz-y]$   \\
            \cline{1-2} Algebra \textbf{U} ($q^l =1 $)&  $K[x_1^l,\cdots,x_n^l]$  \\
            \cline{1-2} Coordinate algebra of the quantum group
            $SL_q(2)$ & $K[xy-quv]$  \\
            \cline{1-2} Coordinate algebra of the quantum group $SL_q(2)$, $ q^l=1 $ & $K[xy-quv,x^{l},y^{l},u^{i}v^{j}]$, with $ i+j=l $  \\
            \cline{1-2} $q$-Heisenberg algebra \textbf{H}$_n(q)$ & $K[C_{i}], C_{i}=(q^{2}-1)x_{i}y_{i}z_{i}-y_{i}^{2}$   \\
            \cline{1-2} $q$-Heisenberg algebra \textbf{H}$_n(q)$, $ q^l=1 $ & $K[x_{i}^{l},y^{l}_{i},z_{i}^{l},C_{i}]$   \\
            \cline{1-2}Quadratic algebras in 3 variables, $ \mathcal{A}_3 $ & $K[\xi_2y-a_2z]$   \\
            \cline{1-2}
    \end{tabular}}
    \caption{Central subalgebras of some bijective skew $PBW$
        extensions.}\label{table4.3}
\end{table}

\subsection{Application to the Zariski cancellation problem}

Now we apply some of the previous results to the Zariski cancellation problem. This is an open
problem of affine algebraic geometry that has recently been formulated and studied by Bell and
Zhang (\cite{BellZhang}) for noncommutative algebras (see also \cite{Zhangetal}, \cite{Zhangetal2},
\cite{Zhangetal3}, \cite{Zhangetal4}, \cite{Wang}).

\begin{definition}
    Let $ A $ be a $ K- $algebra.
    \begin{enumerate}
        \item $ A $ is cancellative if $ A[t]\cong B[t] $ for some $ K- $algebra $ B $ implies that $ A\cong B $.
        \item $ A $ is strongly cancellative if, for any $ d\geq 1 $, the isomorphism $ A[t_1,\cdots,t_d]\cong B[t_1,\cdots,t_d] $ for some $ K- $algebra $ B $ implies that $ A\cong B $.
        \item $ A $ is universally cancellative if, for any finitely generated commutative $ K- $algebra and a domain $ R $ such that $ R/I=K $ for some ideal $ I\subset R $ and any $ K- $algebra $ B $, $ A\otimes R\cong B\otimes R  $ implies that $ A\cong B $.
    \end{enumerate}
\end{definition}

It is clear that universally cancellative implies strongly cancellative, and in turn, strongly
cancellative implies cancellative.
\begin{proposition}
    Let $ K $ be a field and $ A $ be an algebra with $ Z(A)=K $. Then $ A $ is universally cancellative.
\end{proposition}
\begin{proof}
    See \cite{BellZhang}, Proposition 1.3.
\end{proof}
According to the conditions presented in the Table \ref{table4.1} and the previous proposition, we have the following Corollary.
\begin{corollary}\label{zariskipbw}
    The following skew $ PBW $ extensions are universally cancellative, and hence, cancellative:
    \begin{enumerate}
        \item  Weyl algebra $A_n(K)$.
        \item Extended Weyl algebra $B_n(K)$.
        \item Algebra of q-differential operators $D_{q,h}[x,y]$.
        \item  Algebra of shift operators $S_h$.
        \item Mixed algebra $D_h$.
        \item Discrete linear systems $K[t_1,\dotsc,t_n][x_1,\sigma_1]\dotsb[x_n;\sigma_n]$.
        \item Linear partial shift operators $K[t_1,\dotsc,t_n][E_1,\dotsc,E_n]$ .
        \item L.P. Differential operators $K[t_1,\dotsc,t_n][\partial_1,\dotsc,\partial_n]$.
        \item L. P. Difference operators    $K[t_1,\dotsc,t_n][\Delta_1,\dotsc,\Delta_n]$.
        \item L. P. $q$-dilation operators $K[t_1,\dotsc,t_n][H_1^{(q)},\dotsc,H_n^{(q)}]$.
        \item L. P. $q$-differential operators $K[t_1,\dotsc,t_n][D_1^{(q)},\dotsc,D_n^{(q)}]$.
        \item Additive analogue of the Weyl algebra $A_n(q_1,\dotsc,q_n)$.
        \item Multiplicative analogue of the Weyl algebra$\cO_n(\lambda_{ji})$.
        \item $ n- $multiparametric quantum space.
        \item Woronowicz algebra $\cW_{\nu}(\mathfrak{sl}(2,K))$.
        \item  Algebra \textbf{U}.
        \item Differential operators on a quantum space $S_{\textbf{q}}$.
        $D_{\textbf{q}}(S_{\textbf{q}})$.
        \item Particular Witten's Deformation of $\cU(\mathfrak{sl}(2,K)$.
        \item Quantum Weyl algebra of Maltsiniotis.
        $A_n^{\textbf{q},\lambda}$.
        \item Multiparameter Weyl algebra $A_n^{Q,\Gamma}(K)$.
        \item Quantum symplectic space $\cO_q(\mathfrak{sp}(K^{2n}))$.
        \item Jordan plane $ \mathcal{J} $.
        \item Quantum plane.
        \item Quadratic algebras in 3 variables, $ \mathcal{A}_2 $.
    \end{enumerate}
\end{corollary}


\end{document}